\documentclass[11pt]{article}

\usepackage{amsmath,amssymb,graphicx,epsfig,cite,rotating,setspace,amsthm,color}
\newcommand{\beq}{\begin{equation}}  
\newcommand{\eeq}{\end{equation}}  
\newcommand{\bea}{\begin{eqnarray}}  
\newcommand{\eea}{\end{eqnarray}}

\def\ri{\mathrm{i}}

\newcommand{\realpart}[1]{\operatorname{\sf Re}\!\left(#1\right)}

\newcommand{\eq}[2]{\begin{equation}\begin{split}#1\end{split}\label{#2}\end{equation}}

\newcommand{\eqnn}[1]{\begin{equation}\begin{split}#1\end{split}\nonumber\end{equation}}

\newtheorem{theorem}{Theorem}
\numberwithin{theorem}{section}

\newtheorem{lemma}[theorem]{Lemma}

\newtheorem{proposition}[theorem]{Proposition}

\theoremstyle{definition}

\newtheorem{remark}[theorem]{Remark}

\renewcommand\rightmark{Combustion  waves in hydraulically resistant porous media  \hfill\thepage}

\usepackage{authblk}

\begin{document}
\title{The Stability of the $b$-family of Peakon Equations}
\author[1]{Efstathios G. Charalampidis   %
\thanks{echarala@calpoly.edu}}
\affil[1]{Mathematics Department,        %
California Polytechnic State University, %
San Luis Obispo,
CA 93407-0403, USA}

\author[2]{Ross Parker %
\thanks{rhparker@smu.edu}}
\affil[2]{Department of Mathematics, %
Southern Methodist University, Dallas, %
TX 75275, USA}

\author[3]{Panayotis G. Kevrekidis%
\thanks{kevrekid@math.umass.edu}}
\affil[3]{Department of Mathematics %
and Statistics, University of Massachusetts Amherst, %
Amherst, MA 01003-9305, USA}

\author[4]{St\'ephane Lafortune%
\thanks{lafortunes@cofc.edu}}
\affil[4]{Department of Mathematics, %
College of Charleston, Charleston, %
SC 29401, USA}

\maketitle

\begin{abstract} 
In the present work we revisit the $b$-family model of peakon equations,
containing as special cases the $b=2$ (Camassa-Holm) and $b=3$ (Degasperis-Procesi) 
integrable examples. We establish information about the point spectrum 
of the peakon solutions and notably find that for suitably smooth perturbations 
there exists point spectrum in the right half plane rendering the peakons unstable
for $b<1$. We explore numerically these ideas in the realm of fixed-point 
iterations, spectral stability analysis and time-stepping of the model for 
the different parameter regimes. In particular, we identify exact, stationary 
(spectrally stable) lefton solutions for $b<-1$, and for $-1<b<1$, we dynamically 
identify ramp-cliff solutions as dominant states in this regime. We complement 
our analysis by examining the breakup of smooth initial data into stable peakons 
for $b>1$. While many of the above dynamical features had been explored in earlier 
studies, in the present work, we supplement them, wherever possible, with spectral 
stability computations.
\end{abstract} 

\section{Introduction} 
\label{intro}
The family of partial differential equations 
\beq \label{bfamily} 
u_t - u_{xxt} +(b+1)uu_x=bu_x u_{xx} + u u_{xxx}, 
\eeq 
labeled by the parameter $b$, is distinguished by the fact that it includes 
two completely integrable equations, namely the Camassa-Holm equation (the 
case $b=2$~\cite{ch, ch2}), and the Degasperis-Procesi equation (the case 
$b=3$~\cite{dp, dhh}). Each of the two integrable cases has a Lax
pair (and is, thus, solvable via the inverse scattering transform), 
possesses multi-soliton solutions, and a bi-Hamiltonian structure%
~\cite{Ma05,honejpa,ch2,ff}. Furthermore, the cases $b=2,3$ have been singled 
out by various tests of integrability: The Wahlquist-Estabrook prolongation 
method, the Painlev\'e analysis, symmetry conditions, and a test for asymptotic 
integrability~\cite{dp, wanghone, miknov, honep}. 

The Camassa-Holm equation was originally proposed as a model for shallow 
water waves~\cite{ch, ch2}. The results of~\cite{constantin,rossen} (see 
Proposition 2 of~\cite{constantin} and Equation (3.8) of~\cite{rossen}) 
show that, in a model of shallow water, the solution $u$ of Eq.~\eqref{bfamily} 
corresponds to the horizontal component of velocity evaluated at some 
specific level in the cases $b\geq  10/11$ or $b\leq -10$. However, there 
is some debate about the precise range of validity of such models \cite{bm}.

What makes the $b$-family particularly interesting to study from a
mathematical physics viewpoint is that it shares the one-peakon solutions   
\beq \label{peakon}
u=u_0=c\exp (-|x-ct|).
\eeq
that are admitted by the Camassa-Holm, and the Degasperis-Procesi equations.
Indeed, the peakons can solve the following weak formulation of 
Eq.~\eqref{bfamily}
\begin{equation} \label{bfamilyi} \vspace{-.2cm}
u_t=\frac{1}{2}\left(\phi\ast \left[\frac{b-3}{2}u_x^2-%
\frac{b}{2}u^2\right]-u^2\right)_x,\;\;\phi=e^{-|x|},
\end{equation}
where $\ast$ denotes convolution; the fact that $\phi/2$ is a Green's 
function for the operator $1-\partial_x^2$ was used in the reformulation. 
In effect, Eq.~\eqref{bfamilyi} is obtained from Eq.~\eqref{bfamily} 
by factoring out the operator $1-\partial_x^2$.

Moreover, the whole $b$-family possesses $N$-peakon solutions given by  
\beq \label{multipeakon} 
u(x,t) = \sum_{j=1}^N p_j(t)e^{-|x-q_j(t)|}, 
\eeq 
where  the positions $q_j$ and amplitudes $p_j$ are the canonically 
conjugate coordinates and momenta in a finite-dimensional Hamiltonian 
system. In the cases $b=2,3$ this Hamiltonian system is completely 
integrable in the Liouville-Arnold sense~\cite{ch,ch2,dhh}. In the 
general case, the Hamiltonian system does not appear to be integrable%
~\cite{hh}. Recently, the $b$-family was generalized to an equation 
containing two free functions with the property that it also admits 
multi-peakon solutions written as a linear combination of one-peakons~\cite{Anco2019}.  

Another interesting aspect of the $b$-family in the cases $b=2,3$ is 
that they admit smooth multi-soliton solutions on a nonzero background~\cite{Ma05,ch,ch2}. 
In the limit where the background goes to zero, the $N$-soliton solutions 
become the $N$-peakons solution as given in Eq.~\eqref{multipeakon}. For 
general $b$, smooth one-solitons on nonzero background are known to exist~\cite{Guo05}. 
 
The work of ~\cite{Holm03a, Holm03} presented a  numerical study of the
solutions of Eq.~\eqref{bfamily} for different values of $b$. They observed 
that there are three distinct parameter regimes separated by bifurcations 
at $b=1$ and $b=-1$, as follows:

\begin{itemize}
\item {\bf Peakon regime}: For $b>1$, arbitrary initial data asymptotically 
separates out into a number of peakons as $t\to\infty$. 
\item {\bf Ramp-cliff regime}: For $-1<b<1$, solutions behave asymptotically 
like a combination of a ``ramp''-like solution of Burgers equation (proportional
to $x/t$), together with an exponentially-decaying tail (``cliff'').
\item {\bf Lefton regime}: For $b<-1$, arbitrary initial data moves to the 
left and asymptotically separates out into a number of ``leftons'' as 
$t\to\infty$, which are smooth, exponentially localized,
stationary solitary waves. 
\end{itemize}

The behavior observed separately in each of the parameter ranges $b>1$ 
and $b<-1$ can be understood as particular instances of the {\it soliton 
resolution conjecture}~\cite{tao}, a somewhat loosely defined conjecture which states
that for suitable dispersive wave equations, solutions with ``generic'' 
initial data will decompose into a finite number of solitary waves plus a 
radiation part which disperses away. The authors of~\cite{Hone14} provide 
a first step towards explaining this phenomenon analytically in the ``lefton'' 
regime $b<-1$. Indeed, they show that in this parameter range a single 
lefton solution is orbitally stable, by applying the approach of Grillakis, 
Shatah and Strauss in \cite{Grillakis87}. The main ingredients required for
the stability analysis are the Hamiltonian structure and conservation laws 
for Eq.~\eqref{bfamily}. The $b$-family is known to admit a Hamiltonian 
structure and two additional conservation laws~\cite{dhh2}. The lefton 
solutions are a critical point for a functional which is combination of the 
Hamiltonian and a conserved functional.    

In this article, our goal is to study the spectral stability of the peakon 
solutions [cf. Eq.~\eqref{peakon}]. In particular, we are interested in the 
observation made numerically by Holm and Staley in \cite{Holm03a, Holm03} 
that the peakon solutions  become unstable when $b<1$. To do so, in Section 2, 
we state the main analytical results concerning the spectrum associated to 
the eigenvalue problem arising from the linearization of Eq.~\eqref{bfamilyi} 
about the peakon solutions. These analytical results are proven in Section 3. 
The numerical results on the $b$-family [cf. Eq.~\eqref{bfamily}] are presented 
in Section 4. We explore both statically as appropriate, as well as dynamically, 
each of the classes of solutions therein. We examine their existence over parametric 
variations of $b$, when possible/relevant (e.g. for the leftons) we consider their 
stability and we also explore their dynamics (especially for the ramp-cliff waveforms 
for which we cannot identify a reference frame in which they appear as steady). In 
Section 5, we state our conclusions and present directions for future study.

\section{Main Results} 
\label{main_res}
The spectral stability of the peakon solution is explored by first considering 
Eq.~\eqref{bfamilyi} in the co-traveling frame $\xi=x-ct$:
\beq \notag
u_t-cu_\xi=\frac{1}{2}\left(\phi\ast \left[\frac{b-3}{2}u_\xi^2-\frac{b}{2}u^2\right]%
-u^2\right)_\xi,\;\;\phi=e^{-|\xi|}.
\eeq 
Now consider a small perturbation of the peakon $u=u_0$ of the form 
\beq \notag
w(\xi,t)=v(\xi)e^{\lambda t},
\eeq
where $v(\xi)$ stands for the eigenvector associated with the 
eigenvalue $\lambda$. Then, we substitute $u=u_0+w$ into Eq.~\eqref{bfamilyi} 
and linearize by keeping only the first-order terms in $v$. This way, 
we obtain the following eigenvalue problem associated to an integral 
operator ${\mathcal{L}}$
\beq 
\label{eig}
\lambda v={\mathcal{L}}v\equiv \left(\phi\ast \left[\frac{b-3}{2}u_{0}'v'%
-\frac{b}{2}u_0v\right]+(c-u_0)v\right)',\;\;\phi=e^{-|\xi|}.
\eeq
For our analytical study, we are interested in the spectrum of ${\mathcal{L}}$ 
defined above. 

To define an appropriate domain for the operator ${\mathcal{L}}$, we need to 
consider well-posedness of the $b$-family [cf. Eq.~\eqref{bfamilyi}]. The $b$-family 
is known to be well-posed for initial conditions in $H^s(\mathbb{R})$, $s>3/2$~\cite{Gui08,Zhou10,Liu06,Escher08,katelyn,Olver00,Rod93}. 
The peakons in the Camassa-Holm ($b=2$) are proven to be stable in $H^1(\mathbb{R})$%
~\cite{cstrauss} while the ones in the Degasperis-Procesi ($b=3$) in $L^2(\mathbb{R})$~\cite{lin}. 
However, due to the discussion above about well-posedness, the authors of~\cite{cstrauss}
and~\cite{lin} state that their stability results only apply to initial condition that 
are in the subsets $H^s(\mathbb{R})$, $s>3/2$, of $H^1(\mathbb{R})$ (for Camassa-Holm) 
or $L^2(\mathbb{R})$ (for Degasperis-Procesi).

We are thus interested in the orbital stability of the peakon solutions~\eqref{peakon} 
with respect to initial conditions of the form 
\begin{equation}
u(\xi,t=0)=u_0(\xi+\epsilon)+p_0(\xi) \in H^s(\mathbb{R}), s>3/2,
\nonumber
\end{equation}
where $\epsilon$ is introduced to take into account a drift along the translation 
invariance symmetry direction. A necessary condition for the initial conditions 
above to be in $H^s(\mathbb{R})$, for some $s>3/2$ is that $p_0$ be in $H^s(\mathbb{R})$, 
for all $s<3/2$, since $u_0$ itself is in $H^s(\mathbb{R})$, for all $s<3/2$. Thus, 
at the linear level, we will look for eigenvectors of the form
\beq
\label{disc}
v=u_0'(\xi)+p_1(\xi),
\eeq
for $p_1\in H^s(\mathbb{R})$, for all $s<3/2$. However, $v=u_0'$ is discontinuous 
and thus not in the domain of ${\mathcal{L}}$ as defined in Eq.~\eqref{eig} 
due to the term $u_0'v'$. In the next section (see Eq.~\eqref{gen}), we 
will define an extension ${\mathcal{L}}_w$ of the operator ${\mathcal{L}}$ 
given in Eq.~\eqref{eig} that admits discontinuous functions in its domain. 
For ${\mathcal{L}}_w$, we will be interested in eigenfunctions in the set 
\beq
\label{Dom}
A=\left\{u_0'(\xi)+p_1(\xi)\,\bigg|\, p_1\in H^s(\mathbb{R}),{\mbox{ for all }}s<3/2\right\}.
\eeq
The extension ${\mathcal{L}}_w$ is not a closed operator on $L^2(\mathbb{R})$ and 
thus its resolvent set is automatically empty (see for example~\cite{Schmudgen12}). 
However, we show that ${\mathcal{L}}_w$ is closed on the Banach space $L^2(\mathbb{R})\cap C_{\rm{d}}(\mathbb{R})$ 
(see Lemma~\ref{closed}), where $C_{\rm{d}}(\mathbb{R})$ is the set of bounded functions 
that are continuous except at the origin, where the functions are allowed to have a 
finite jump discontinuity (see Eq.~\eqref{Cd}). In Section \ref{SP}, we prove the 
following theorem about the point spectrum of ${\mathcal{L}}_w$:
\begin{theorem}
\label{im}
The linear operator $\mathcal{L}_w$ defined in Eq.~\eqref{gen} 
is closed on $L^2(\mathbb{R})\cap C_{\rm{d}}(\mathbb{R})$ and its point spectrum 
consists of the origin $\lambda=0$ and, if $b<2$, of the two bands defined by 
$0< |\realpart{\lambda}|<  c(2-b)$. If we restrict the eigenfunction to the set 
$A$ defined in Eq.~\eqref{Dom}, the band of point spectrum is reduced to 
$0< |\realpart{\lambda}|\leq c(1-b)$ when $b<1$.
\end{theorem}

The second statement within the Theorem~\ref{im} above provides an explanation 
for the observation made numerically by Holm and Staley in~\cite{Holm03a, Holm03} 
that the peakon solutions are unstable when $b<1$. {Remark~\ref{Re} illustrates 
the fact that if the chosen space is made of functions with more regularity than the ones in $L^2(\mathbb{R})\cap C_{\rm{d}}(\mathbb{R})$, the 
width of the band obtained in the first part of the Theorem \ref{im} decreases. Actually, Eq.~\eqref{Ul} shows that with enough regularity,
the two bands in Theorem \ref{im} can be made as close as one wants to each other.}   
{A spectrum consisting of a strip about the imaginary axis also occurs in the study 
of the peaked periodic wave of both versions of the reduced Ostrovsky equations~\cite{Pelinovsky2019}. 
Although our solutions are not periodic, the nature of the result is similar.}

In what follows, we also explore numerically the waveforms of the model
for different values of $b$. We identify the leftons as stationary
solutions for $b<-1$ and illustrate their potential spectral stability.
We dynamically examine the ramp-cliff solutions for $-1<b<1$ and show
that progressively refined computations (involving more modes) suggest
that the ramp-cliff solutions deform into emitting peakons close to
$b=1$ (the more refined the computations, the closer to $b=1$ this
phenomenology arises). Beyond $b=1$ in line with the theory above,
we find that initial data breaks up spontaneously into arrays of
peakons that appear to be dynamically robust. A complementary perspective
that we provide to avoid issues with the discontinuity of the peakons
involves the stability analysis of the solutions of non-vanishing
background, as they approach the vanishing background (i.e., peakon) limit.

\section{Computation of the point spectrum}
\label{SP}
 
In this section, we compute the point spectrum of Eq.~\eqref{eig} for values 
of $\lambda$ such that $\realpart{\lambda}\geq 0$. The case where $\realpart{\lambda}<0$ 
can be obtained from the spectrum of the right side of the complex plane by 
making the observation that if $v(\xi)$ solves the eigenvalue problem of Eq.~\eqref{eig} 
for a given value of $\lambda=\lambda_0$, then $v(-\xi)$ solves that same 
eigenvalue problem with $\lambda=-\lambda_0$ as the corresponding value of 
$\lambda$. 

We first show that the operator $\mathcal{L}$ defined in Eq.~\eqref{eig} does 
not have continuous eigenvectors.

\begin{proposition}
The eigenvalue problem of Eq.~\eqref{eig} does not have solutions in $H^1(\mathbb{R})$.
\end{proposition}
\begin{proof}
Consider the problem of Eq.~\eqref{eig} for $\xi<0$. We apply the operator 
$1-\partial_\xi^2$ to obtain the following differential equation
\beq 
\label{left}
\lambda(v-v'')+c\left(v''-v+be^\xi v+(1-b)e^\xi v'-e^\xi v''\right)'=0,
\eeq 
where we have used the fact that $\phi/2$ is a Green's function for the 
operator $1-\partial_\xi^2$. It turns out there are two solutions to 
Eq.~\eqref{left} converging as $\xi\rightarrow -\infty$, one as $e^\xi$ 
and one as $e^{\lambda\xi/c}$. There is also a solution diverging as 
$e^{-\xi}$. These decay and growth rates are found by solving the constant 
coefficient asymptotic system obtained by applying the limit $\xi\rightarrow -\infty$ 
to Eq.~\eqref{left}. Actually, Eq.~\eqref{left} admits the two explicit 
solutions $v_{l1}=e^{-\xi}$ and $v_{l2}=e^{\xi}$. 

It should be noted that $\xi=0$ in Eq.~\eqref{left} is a regular singular 
point with exponents $r_1=0,\,r_2=1$, and $r_3=-\lambda/c+2-b$. A third
solution $v_{l3}$ linearly independent that is not singular at $\xi=0$ can 
be found if we assume $\realpart{r_3}>0$. It can be defined by its series 
expansion about $\xi=0$
\beq\label{vl3}
v_{l3}=
\begin{cases}
|\xi|^{r_3}+\mathcal{O}\left(|\xi|^{r_3+1}\right)\;\;{\mbox{if}}\;\;r_3\neq 1,\\
\xi\ln|\xi|+\mathcal{O}\left(\xi^{2}\ln|\xi|\right)\;\;{\mbox{if}}\;\;r_3=1,
\end{cases}
\eeq
where
\eq{r_3=-\lambda/c+2-b.}{r3}
Upon adding an appropriate multiple of $e^\xi-e^{-\xi}$ to $v_{l3}$, the 
solution converges as $\xi\rightarrow -\infty$ and is zero at $\xi=0$.  
Indeed, let $v={F}$ be the solution of Eq.~\eqref{left} defined as
\beq
\label{F}
F=v_{l3}-C(e^{-\xi}-e^{\xi}),\;\;{\mbox{where}}\;\;%
C=\lim_{\xi\rightarrow -\infty} e^{\xi} v_{l3}(\xi).
\eeq
Then $F$ is such that $F(0)=0$ and $F\rightarrow 0$ as $\xi\rightarrow -\infty$. 
  
Applying the operator $1-\partial_\xi^2$ to the eigenvalue problem of Eq.~\eqref{eig} for 
$\xi>0$, one obtains a differential equation with only one converging solution as 
$\xi\rightarrow \infty$ given by $v=e^{-\xi}$. Hence, to look for a solution 
to Eq.~\eqref{eig} that is bounded, in the case where $\realpart{r_3}>0$, one considers
\beq
\notag
v=
\begin{cases}
\displaystyle{c_0 e^{-\xi}{\mbox{ for }}\xi>0,} \\
\displaystyle{c_1 F(\xi)+c_2 e^\xi+{\mbox{ for }}\xi<0.}
\end{cases}
\eeq
Since we look for continuous solution, and because $F(0)=0$, we need to 
take $c_0=c_2$. The most general ansatz in this case is
\beq
\label{hypcont}
v_c=c_1v_1+c_2v_2,{\mbox{ where }}v_1\equiv H(-\xi)F(\xi){\mbox{ and }}v_2\equiv e^{-|\xi|},
\eeq
with $H$ being the Heaviside function. 
\begin{lemma}
If we substitute $v=v_c$ into 
Eq.~\eqref{eig}, one obtains
\beq
\label{cont}
{\mathcal{L}}v_c-\lambda v_c=\left(c_1 \frac{3c(b-2)}{2}\int_{-\infty}^0 e^{2\xi'}F(\xi') d\xi'
+c_2\left(c\,{\mbox{sgn}}{(\xi)}-\lambda\right)\right)e^{-|\xi|}.
\eeq
\end{lemma}
\begin{proof}
By substituting $v=v_2=e^{-|\xi|}$ into Eq.~\eqref{eig}, it is a straightforward 
computation to find that
\eq{
{\mathcal{L}}v_2-\lambda v_2=\left(c\,{\mbox{sgn}}{(\xi)}-\lambda\right)e^{-|\xi|}.
}{res1}
When substituting $v=v_1=H(-\xi)F(\xi)$, there are two cases to consider: 
$\xi>0$ and $\xi<0$. If $\xi>0$, we substitute $v_1=H(-\xi)F(\xi)$ into 
Eq.~\eqref{eig} and obtain
\eq{
{\mathcal{L}}v_1-\lambda v_1&=\left(\phi\ast \left[\frac{b-3}{2}u_{0}'v_1'%
-\frac{b}{2}u_0v_1\right]+(c-u_0)v_1\right)'-\lambda v_1 \\ 
&=c\left(\int_{-\infty}^0 e^{-|\xi-\xi'|}e^{\xi'}\left[\frac{b-3}{2}  F'({\xi'})-\frac{b}{2}F({\xi'})\right]d{\xi'}\right)'\\
&=c\left(e^{-\xi}\int_{-\infty}^0 e^{2{\xi'}}\left[\frac{b-3}{2}  F'({\xi'})-\frac{b}{2}F({\xi'})\right]d{\xi'}\right)'\\
&=e^{-\xi} \left(\frac{3c(b-2)}{2}\int_{-\infty}^0 e^{2{\xi'}}F({\xi'}) d{\xi'}\right){\mbox{  for  }}\xi>0.
}{res2}
For $\xi<0$, one uses the fact that $F$ satisfies \eqref{left} itself obtained
by the application of the operator $1-\partial_\xi^2$ on the eigenvalue problem 
\eqref{eig}, that is
$$
(1-\partial_\xi^2)\left({\mathcal{L}}v_1-\lambda v_1\right)=0{\mbox{  for  }}\xi<0.
$$
This implies that ${\mathcal{L}}v_1-\lambda v_1$ is a linear combination
of $e^\xi$ and $e^{-\xi}$ for $\xi<0$. Since $F$ converges to zero as 
$\xi\rightarrow -\infty$, we have that
$$
{\mathcal{L}}v_1-\lambda v_1=Be^\xi, \,B=\mathrm{const.},\,{\mbox{  for  }}\xi<0.
$$
Furthermore, it can be checked that ${\mathcal{L}}v_1-\lambda v_1$ is 
continuous at $\xi=0$ due to the fact  that $v_1$ is a a continuous function 
such that $v_1(0)=0$. This check is done using the expression obtained below 
in Eq.~\eqref{gend} for the extension $\mathcal{L}_w$ of $\mathcal{L}$. The 
discontinuous part of ${\mathcal{L}}_wv-\lambda v$ on the second line of Eq.~\eqref{gend} 
is zero if $v(0)=0$. Thus by continuity, with Eq.~\eqref{res2} that 
$B= \left(\frac{3c(b-2)}{2}\int_{-\infty}^0 e^{2{\xi'}}F({\xi'}) d{\xi'}\right)$ 
and thus
\eq{{\mathcal{L}}v_1-\lambda v_1=e^{-|\xi|} \left(\frac{3c(b-2)}{2}\int_{-\infty}^0 e^{2\xi'}F(\xi') d\xi'\right).
}{res3} 
Then \eqref{cont} follows from \eqref{res1} and \eqref{res3}.
\end{proof}

To prove that Eq.~\eqref{eig} does not have continuous solutions, we need to 
prove that the right-hand-side (RHS) of Eq.~\eqref{cont} cannot be zero for 
any $c_1$ and $c_2$. For the RHS of Eq.~\eqref{cont} to be zero, $c_2$ must 
be zero since the expression it multiplies is discontinuous. The statement 
of the proposition then stems from the following lemma.
\begin{lemma} Assume that $r_3=2-b-\lambda/c$ has a positive real part. 
Let $F$ be the unique (up to multiplication by a scalar) solution to 
Eq.~\eqref{left} such that $F(0)=0$ and $F\rightarrow 0$ as $\xi\rightarrow -\infty$. 
Then 
\begin{equation}
\int_{-\infty}^0 e^{2\xi}F(\xi) d\xi\neq 0.
\nonumber
\end{equation}
\end{lemma}
\begin{proof}
We make the substitution $v=e^\xi u$ into Eq.~\eqref{left} to get a second-order 
equation for $u'$ since $v=e^\xi$ solves Eq.~\eqref{left}. The new equation admits 
$u'=e^{-2\xi}$ as a solution. We then make the substitution $u'=e^{-2\xi}w$ 
and get a first-order equation for $w'$ whose solution is
\beq
\label{wp}
w'=\widetilde{B}\frac{e^{(\lambda/c+1)\xi}}{(e^\xi-1)^{\lambda/c-b}}, \, \widetilde{B}=\mathrm{const.}
\eeq

This way, we have that $e^{2\xi}v'=e^{2\xi}(e^\xi u)'=e^{2\xi}(e^\xi u+e^\xi u')=e^{2\xi}v+e^{\xi}w$, 
thus implying
\beq
\label{w}
e^{2\xi}v'-e^{2\xi}v=e^{\xi}w.
\eeq
We substitute $v=F$ in the equation above and integrate both sides from 
$-\infty$ to $0$. We integrate by parts the first term of the left-hand-side 
(LHS), using the fact that $F(0)=0$, and obtain
\begin{equation}\notag
-\frac{3}{2}\int_{-\infty}^0e^{2\xi}F(\xi)d\xi=\int_{-\infty}^0e^{\xi}w(\xi)d\xi.
\end{equation}
Since the sign of $w'$ never changes by Eq.~\eqref{wp} and $w\rightarrow 0$ as 
$\xi\rightarrow -\infty$ by Eq.~\eqref{w}, we have that $w$ never changes sign 
for $\xi<0$ and the integrals above are both nonzero.
\end{proof}
\end{proof}

We now want to consider solutions to the eigenvalue problem of [cf. Eq.~\eqref{eig}] 
admitting a discontinuity at the origin such as in Eq.~\eqref{disc}. To do so, we 
introduce an extension of ${\mathcal{L}}$ as defined in Eq.~\eqref{eig}. We first 
consider ${\mathcal{L}}$ in the case $\xi<0$, which we denote by ${\mathcal{L}}_-$:
\eqnn{
{\mathcal{L}}_-v=& \frac{\rm{d}}{\rm{d}\xi}\left(\int_{-\infty}^{\infty}%
e^{-|\xi-\xi'|}\left[\frac{b-3}{2}u_{0}'(\xi')v'(\xi')\right]d\xi'-\frac{b}{2}\phi\ast %
u_0v+(c-u_0)v\right)\\
=&\frac{c(b-3)}{2}\frac{\rm{d}}{\rm{d}\xi}\Bigg(e^{-\xi}\int_{-\infty}^{\xi}e^{2\xi'}v'(\xi')d\xi'
+e^{\xi}\int_{\xi}^{0}v'(\xi')d\xi'\\
&-e^{\xi}\int_{0}^{\infty}e^{-2\xi'}v'(\xi')d\xi'\Bigg)%
+\frac{\rm{d}}{\rm{d}\xi}\left(-\frac{b}{2}\phi\ast (u_0 v)+(c-u_0)v\right),
}
where we used the fact that $u_0'=-c\,{\mbox{sgn}}(\xi)e^{-|\xi|}$. Then, we use 
integration by parts to eliminate $v'$ and obtain
\begin{align}
{\mathcal{L}}_-v=& \frac{\rm{d}}{\rm{d}\xi}\Bigg\lbrace c(3-b)%
\Bigg(e^{\xi}\int_{0}^{\infty}e^{-2\xi'}v(\xi')d\xi'+e^{-\xi}%
\int_{-\infty}^{\xi}e^{2\xi'}v(\xi')d\xi' \nonumber \\ %
&-\frac{e^\xi\left(v_0^++v_0^-\right)}{2}\Bigg)%
-\frac{b}{2}\phi\ast (u_0 v)+(c-u_0)v\Bigg\rbrace,
\label{plus}
\end{align}
where 
\begin{equation}
v_0^\pm\equiv \lim_{\xi\rightarrow 0^\pm} v(\xi). 
\nonumber
\end{equation}

In the case of $\xi>0$, we get
\begin{align}
{\mathcal{L}}_+v=& \frac{\rm{d}}{\rm{d}\xi}\Bigg\lbrace c\left(3-b\right)%
\Bigg(e^{-\xi}\int_{-\infty}^{0}e^{2\xi'}v(\xi')d\xi'+e^{\xi}%
\int_{\xi}^{\infty}e^{-2\xi'}v(\xi')d\xi'\nonumber \\
&-\frac{e^{-\xi}\left(v_0^++v_0^-\right)}{2}\Bigg)%
-\frac{b}{2}\phi\ast (u_0 v)+(c-u_0)v
\Bigg\rbrace.
\label{moins}
\end{align}
Thus the extension of the operator ${\mathcal{L}}$ (from Eq.~\eqref{eig})
reads
\beq
\notag
{\mathcal{L}}_w\equiv 
\begin{cases}
\displaystyle{{{\mathcal{L}}_+}}\;\;{\mbox{for}}\;\;\xi>0,\\
\displaystyle{{{\mathcal{L}}_-}}\;\;{\mbox{for}}\;\;\xi<0,
\end{cases}
\eeq
which has a larger domain, and $\mathcal{L}v=\mathcal{L}_wv$ if 
$v\in {\rm{Dom}}\left(\mathcal{L}\right)\subseteq L^2(\mathbb{R})$. 
Indeed, the domain of $\mathcal{L}$ (from the definition given in Eq.~\eqref{eig}), 
is restricted to $H^1(\mathbb{R})$,  while the operator $\mathcal{L}_w$ 
(from Eqs.~\eqref{plus} and~\eqref{moins}) admits functions that have a 
finite-jump discontinuity at $\xi=0$. Furthermore, it is straightforward 
to verify that $u_0'$ is in the kernel of ${\mathcal{L}}_w$, i.e.
\beq
\label{zero}
{\mathcal{L}}_wu_0'=0.
\eeq
This is done by substituting $v=-ce^{-\xi}u_0'=-c\,{\mbox{sgn}}(\xi)e^{-|\xi|}$
into \eqref{plus} and \eqref{moins}.

We can group terms in Eqs.~\eqref{plus} and~\eqref{moins} as
\begin{align}
{\mathcal{L}}_w v=&\frac{\rm{d}}{\rm{d}\xi}%
\Bigg\lbrace%
c(3-b)%
\Bigg(
e^{-\xi}\int^{\min(\xi,0)}_{-\infty}e^{2\xi'}v(\xi')d\xi'%
+e^{\xi}\int^{\infty}_{\max(\xi,0)}e^{-2\xi'}v(\xi')d\xi'%
\nonumber \\
&-\frac{e^{-|\xi|}\left(v_0^++v_0^-\right)}{2}\Bigg) %
-\frac{b}{2}\phi\ast (u_0 v)+(c-u_0)v\Bigg\rbrace,
\label{gen}
\end{align}
and rewrite $\mathcal{L}_w$ (by applying the derivative
operator in Eq.~\eqref{gen}) as
\eq{
{\mathcal{L}}_w v=&c(3-b)\left(e^{\xi}\int^{\infty}_{\max(\xi,0)}%
e^{-2\xi'}v(\xi')d\xi'-e^{-\xi}\int^{\min(\xi,0)}_{-\infty}e^{2\xi'}%
v(\xi')d\xi'\right)-\frac{b}{2}\phi'\ast (u_0 v)  \\
&+c(3-b){\rm{sgn}}(\xi)e^{-|\xi|}\left(\frac{\left(v_0^++v_0^-\right)}{2}-v\right)+\left((c-u_0)v\right)'.
}{gend}

In order to show that ${\mathcal{L}}_w$ is not closable on $L^2(\mathbb{R})$, 
it suffices to show that there is a sequence $v_n$ converging to zero, 
while $\mathcal{L}_wv_n$ does not~\cite{Evans87,Schmudgen12}. We choose the 
sequence of bump functions defined as
\eqnn{
v_n\equiv
\begin{cases}
\displaystyle{\exp{\left(\frac{1}{n^2\xi^2-1}\right)}{\mbox{ for }}|\xi|<1/n,}\\ \\
\displaystyle{0{\mbox{ otherwise.}}}
\end{cases}
}
Clearly, $v_n$ converges to 0 in $L^2(\mathbb{R})$, and all the terms in 
Eq.~\eqref{gen} do also except for the third one since $(v_{n0}^++v_{n0}^-)$ 
converges to $2e^{-1}$. 

In order to define a space on which ${\mathcal{L}}_w$ is closed, we 
first introduce the following subspace of $L^\infty(\mathbb{R})$ made 
of functions that are continuous everywhere except at $\xi=0$. More 
precisely
\eq{C_{\rm{d}}(\mathbb{R})=
\left\{v\in L^\infty(\mathbb{R}) \Big| v\in C_b(\mathbb{R}{\mbox{\textbackslash}}\{0\})%
{\mbox{ and }} \lim_{\xi\rightarrow 0^\pm}v=v_0^\pm{\mbox{ exist}}\right\}.
}{Cd}
The set $C_{\rm{d}}(\mathbb{R})$ with the $L^\infty(\mathbb{R})$ norm is 
a Banach space, since it is isomorphic to the direct sum 
$C_b((-\infty,0]) \oplus C_b([0,\infty))$ equipped 
with the norm $\max\left(\|v\|_{L^\infty((-\infty,0])},\,\|v\|_{L^\infty([0,\infty))}\right)$.

The operator ${\mathcal{L}}_w$ is defined almost everywhere on 
$L^2(\mathbb{R})\cap C_{\rm{d}}(\mathbb{R})$, and thus we have
the following lemma.

\begin{lemma}
\label{closed}
The operator $\mathcal{L}_w$ is closed on $L^2(\mathbb{R})\cap C_{\rm{d}}(\mathbb{R})$.
\end{lemma}

\begin{proof}
We first consider the operator $\widetilde{\mathcal{L}}_w$ defined by
\eq{
\widetilde{{\mathcal{L}}}_w v=\left((c-u_0)v\right)'-c(3-b){\rm{sgn}}(\xi)e^{-|\xi|}v.
}{gendt}
We prove that ${{\mathcal{L}}}_w-\widetilde{{\mathcal{L}}}_w$ is compact on 
$L^2(\mathbb{R})\cap C_{\rm{d}}(\mathbb{R})$ followed by the use of Theorem 1.11 %
of~\cite{Kato}\footnote{It states that if an operator is closed, then so is any 
relatively compact perturbation of that operator.}.

We first prove that each term on the first line of Eq.~\eqref{gend} is compact 
on $L^2(\mathbb{R})\cap C_{\rm{d}}(\mathbb{R})$ by proving they are compact 
on both $L^2(\mathbb{R})$ and $L^\infty(\mathbb{R})$. They are compact on 
$L^2(\mathbb{R})$ because each term on the first line of Eq.~\eqref{gend} can 
be written as an integral operator for some kernel $K\in L^1(\mathbb{R}^2)\cap L^2(\mathbb{R}^2)$.
As such, each of those terms defines a Hilbert-Schmidt integral operator, 
known to be compact (see~\cite{Renardy}, p.~262). For example, the first term 
in parentheses in Eq.~\eqref{gend} corresponds to the kernel
\eq{
K_1=
\begin{cases}
\displaystyle{e^{\xi-2\xi'}\;\; \rm{for}\;\;\xi'>\max{(\xi,0)}},\\
\displaystyle{0\;\;{\rm{otherwise}}.}
\end{cases}
}{Kernelexample}
To prove the integral terms in Eq.~\eqref{gend} are compact on $L^\infty(\mathbb{R})$, 
we use the Corollary 5.1 of~\cite{Eveson}, giving the conditions on the kernel of 
an integral operator for it to be compact on $L^\infty(\mathbb{R}^n)$. Those conditions 
reduce to the following in the case of $n=1$ dimension.

Assume that there is a constant $M$ such that for almost all $\xi\in\mathbb{R}$, 
$K(\xi,\cdot)\in L^1(\mathbb{R})$ and $\|K(\xi,\cdot)\|_1\leq M$. Then the 
operator is compact if and only if for any $\varepsilon>0$ there exist 
$\delta>0$ and $R>0$ such that for almost all $\xi\in\mathbb{R}$ and all 
$h\in(-\delta,\delta)$ we have
\eq{\int_{\mathbb{R}\setminus(-R,R)}|K(\xi,\xi')|\mathrm{d}\xi'<\varepsilon}{c1}
and
\eq{\int_{\mathbb{R}}|K(\xi,\xi'+h)-K(\xi,\xi')|\mathrm{d}\xi'<\varepsilon.}{c2}

To check those conditions on the kernel $K_1$ defined in Eq.~\eqref{Kernelexample}, 
we compute its $L^1$ norm and find that it is bounded by $M=1/2$. We can also compute
the integral in Eq.~\eqref{c1} and find that is is bounded by $e^{-R}/2$. Finally, 
the integral in Eq.~\eqref{c2} is found to be bounded by $1-e^{-2|h|}$. The conditions 
of compactness on $L^\infty(\mathbb{R})$ can also be verified straightforwardly for 
the two other terms of the first line of Eq.~\eqref{gend}. For the second term, we 
have  
\eqnn{
K_2=
\begin{cases}
\displaystyle{e^{-\xi+2\xi'}\;\; \rm{for}\;\;\xi'<\min{(\xi,0)}},\\
\displaystyle{0\;\;{\rm{otherwise}}.}
\end{cases}
}
The condition on the $L^1$ norm, and conditions \eqref{c1} and \eqref{c2} are verified 
based on the fact that
$$
K_2(\xi,\xi')=K_1(-\xi,-\xi').
$$
For the third term in Eq.~\eqref{gend}, we have the kernel 
$$
K_3=-K_{3a}+K_{3b},
$$
where
\eqnn{
K_{3a}=
\begin{cases}
\displaystyle{e^{-\xi+\xi'-|\xi'|}\;\; \rm{for}\;\;\xi>\xi'}\\
\displaystyle{0\;\;{\rm{otherwise}}}
\end{cases},\;\;
K_{3b}=
\begin{cases}
\displaystyle{e^{\xi-\xi'-|\xi'|}\;\; \rm{for}\;\;\xi<\xi'}\\
\displaystyle{0\;\;{\rm{otherwise}}}
\end{cases}.
}
Since $K_{3b}(\xi,\xi')=K_{3a}(-\xi,-\xi')$, we only have to verify 
the conditions for $K_{3a}$. An integral computation shows that the 
$L^1$ norm of $K_{3a}$ is bounded by $1/2+1/e$. Furthermore, another 
integral computation shows that the integral in \eqref{c1} is bounded 
by $e^{-R}$. For the integral in \eqref{c2}, one has to consider several 
cases depending on the signs of $\xi$, $h$, and $\xi-h$. In each case, 
one finds that the integral is bounded by an expression that goes to 
zero as $h\rightarrow 0$. Note that it would have been sufficient to show 
that the integral operators on the first line of Eq.~\eqref{gend} are 
continuous in order to prove the lemma. However, compactness is the stronger 
property that may be useful in the future to obtain the full spectrum.

We now prove that the remaining term of ${{\mathcal{L}}}_w-\widetilde{{\mathcal{L}}}_w$
defined by 
$$
Av\equiv \frac{c(3-b)\left(v_0^++v_0^-\right)}{2}{\rm{sgn}}(\xi)e^{-|\xi|},
$$
is compact on $L^2(\mathbb{R})\cap C_{\rm{d}}(\mathbb{R})$. To prove compactness, 
we need to take a bounded sequence $\{v_n\}$ of $C_{\rm{d}}(\mathbb{R})$ and prove 
that $\{Av_n\}$ has a Cauchy subsequence. The boundedness of  $\{v_n\}$ on 
$C_{\rm{d}}(\mathbb{R})$ implies the boundedness of  $\{v_{n0}^\pm\}$, with 
$v_{n0}^\pm\equiv \lim_{\xi\rightarrow 0^\pm}v_n$. Thus, the sequence 
$\{v_{n0}^++v_{n0}^-\}$ contains a Cauchy subsequence $\{v_{n_i0}^++v_{n_i0}^-\}$. 
With the $L^\infty(\mathbb{R})$ norm we have
$$
\| Av_{n_i}-Av_{n_j}\|_{L^\infty(\mathbb{R})}=%
\frac{c(3-b)}{2}\Big|(v_{n_i0}^++v_{n_i0}^-)-(v_{n_j0}^++v_{n_j0}^-)\Big|,
$$
and with the $L^2(\mathbb{R})$ norm
\eqnn{
\| Av_{n_i}-Av_{n_j}\|_{L^2(\mathbb{R})}&= %
\frac{c(3-b)}{2}\Big|v_{{n_i}0}^++v_{{n_i}0}^--v_{n_j0}^+-v_{n_j0}^-\Big|\,%
\|e^{-|\xi|}\|_{L^2(\mathbb{R})}\\
&=\frac{c(3-b)}{2}\Big|(v_{{n_i}0}^++v_{{n_i}0}^-)-(v_{n_j0}^++v_{n_j0}^-)\Big|.
}
Thus, the sequence $\{Av_{n_i}\}$ is a Cauchy subsequence of $\{Av_{n}\}$ on 
both $C_{\rm{d}}(\mathbb{R})$ and $L^2(\mathbb{R})$. We conclude that $A$ is 
compact on $L^2(\mathbb{R})\cap C_{\rm{d}}(\mathbb{R})$.

It now suffices to prove that $\widetilde{{\mathcal{L}}}_w$ defined in 
Eq.~\eqref{gendt} is closed. Assume we have a converging sequence in the 
domain of $\mathcal{L}_w$, $v_n\rightarrow {v}$ such that $\widetilde{{\mathcal{L}}}_w v_n$ 
also is converging. We need to show that $\widetilde{{\mathcal{L}}}_w v_n\rightarrow \widetilde{{\mathcal{L}}}_w {v}$.
The convergence of the term $-c(3-b){\rm{sgn}}(\xi)e^{-|\xi|}v_n$ 
to $-c(3-b){\rm{sgn}}(\xi)e^{-|\xi|}v$ in $L^\infty(\mathbb{R})\cap L^2(\mathbb{R})$ 
is immediate. For the term $\left((c-u_0)v\right)'$, the $L^2(\mathbb{R})$ convergence of 
$v_n$ to ${v}$ implies the $L^2(\mathbb{R})$ of $(c-u_0)v$ to $(c-u_0)\widetilde{v}$. 
Furthermore, since $\left((c-u_0)v_n\right)'$ itself is convergent in $L^2(\mathbb{R})$, 
it converges to $\left((c-u_0){v}\right)'$, by definition of convergence on 
$H^1(\mathbb{R})$. The convergence of the term $\left((c-u_0)v_n\right)'$ in the sup 
norm to $\left((c-u_0){v}\right)'$ follows from the fact that  $\widetilde{\mathcal{L}}_w v_n$ 
converges in both $L^2(\mathbb{R})$ and $C_{\rm{d}}(\mathbb{R})$ to the same 
function, by the definition of the norm on $L^2(\mathbb{R})\cap C_{\rm{d}}(\mathbb{R})$ 
as being the maximum of the two norms. 
\end{proof}

We are now ready to prove Theorem \ref{im}.
\begin{proof}
We first compute the point spectrum of $\mathcal{L}_w$ associated with the eigenvalue 
problem of Eq.~\eqref{gen}. The most general candidate for a discontinuous solution at 
$\xi=0$ is given by Eq.~\eqref{hypcont}. Without loss of generality, we choose
\eqnn{
c_0&=c\,\widetilde{c}_0-\widetilde{c}_2,\\
c_2&=c\,\widetilde{c}_0+\widetilde{c}_2.
}
Hence, the most general ansatz for a solution in $L^2(\mathbb{R})$ 
in this case is
\eqnn{
v_d=\widetilde{c}_0u_0'+c_1v_1+\widetilde{c}_2 v_2,{\mbox{ where }}v_1=%
H(-\xi)F(\xi){\mbox{ and }}v_2=e^{-|\xi|},
}
where we used the fact that $u_0'=-c\,{\mbox{sgn}}(\xi)e^{-|\xi|}$.
Computing ${\mathcal{L}}_w v_d$, using Eqs.~\eqref{cont} and \eqref{zero}, 
we find
\begin{align}
{\mathcal{L}}_w v_d-\lambda v_d=&e^{-|\xi|}\Bigg ( %
c_1 \frac{3c(b-2)}{2}\int_{-\infty}^0 e^{2\xi}F(\xi) %
d\xi\nonumber +\widetilde{c}_2\left(c\,{\mbox{sgn}}{(\xi)}-\lambda\right)\\
&+
\widetilde{c}_0\,c\,\lambda %
{\mbox{sgn}}(\xi)
\Bigg ).
\notag
\end{align}
Thus, $v_d$ is a solution given that $\widetilde{c}_2=-\widetilde{c}_0\lambda$ 
and $c_1$ is chosen such that $c_1=\frac{2\widetilde{c}_2}{{3c(b-2)}}%
\lambda\left(\int_{-\infty}^0 e^{2\xi}F(\xi) d\xi\right)^{-1}$. From the 
expansions given in Eq.~\eqref{vl3}, if we add the restriction that $F$ be 
in $C_{\rm{d}}(\mathbb{R})$, we have that 
\eq{\realpart{r_3}=2-b-\realpart{\lambda}/c> 0,}{Rr3c} 
i.e.~any $\lambda$ satisfying $0<\realpart{\lambda}< c(2-b)$ is in the point 
spectrum.

The following lemma proves the second part of Theorem \ref{im}.
\begin{lemma}
\label{l35}
The function  $v_1=H(-\xi)F(\xi)$, where $F$ solves Eq.~\eqref{left} 
such that $F(0)=0$ and $F(\xi)\rightarrow 0$ as $\xi\rightarrow -\infty$, is in 
$H^s(\mathbb{R})$, for all $s<3/2$ if and only if $r_3$ given by Eq.~\eqref{r3} 
satisfies $\realpart{r_3}=2-b-\realpart{\lambda}/c\geq 1$.
\end{lemma}
\begin{proof}
Because the series expansion of Eq.~\eqref{vl3} admits a different form for 
$r_3=1$ and $r_3\neq 1$, we treat the two cases separately starting with $r_3\neq 1$. 
In view of the definition of $F$ from Eqs.~\eqref{vl3} and~\eqref{F}, if we 
require $F$ to be in $H^1(\mathbb{R})$, it implies that $\realpart{r_3}>1/2$. 
We write $F$ as $F=\widetilde{F}+H(-\xi)\,|\xi |^{r_3}e^\xi $. Because $\realpart{r_3}>1/2$, 
we have that $\widetilde{F}$ is in $H^2(\mathbb{R})$. It thus suffices to show 
that the function 
\eq{
T(\xi)\equiv H(-\xi)\,|\xi |^{r_3}e^\xi =
\left\{
\begin{array}{l}
|\xi|^{r_3}e^{\xi},\;\;\xi<0,\\
\\
0,\;\;\xi>0,
\end{array}
\right.
}{T}
is in $H^s(\mathbb{R})$ for all $s<3/2$ if only if $\realpart{r_3}\geq 1$. As an 
example, if we use $r_3=1$ in Eq.~\eqref{T}, then the Fourier transform of $T$ is 
given by
$$
\widehat{T}(w)=\frac{1}{(1-\ri w)^2}.
$$ 
Recall that the condition for $T$ to be in $H^s(\mathbb{R})$ is that $(1+w^2)^{s/2}\hat{T}(w)$ 
be in $L^2(\mathbb{R})$ \cite{Hitch}. This condition is satisfied if and only if $s<3/2$. The 
same condition on $s$ is obtained if we use $r_3$ such that $\realpart{r_3}=1$ and thus 
it is clear from $T$ given by Eq.~\eqref{T} that it will be in $H^s(\mathbb{R})$ for all 
$s<3/2$ if and only if $\realpart{r_3}\geq 1$. It can also be checked directly 
by the following expression giving the Fourier transform of $T$ [cf. Eq.~\eqref{T}]
for general values of $r_3$
$$
\widehat{T}(w)=\frac{\ri \Gamma(r_3+1)}{(1-\ri w)^{r_3+1}},
$$
where $\Gamma$ is the Gamma function. As per the case for $r_3=1$, from the 
second line of Eq.~\eqref{vl3}, we consider the function
$$
T_1(\xi)\equiv H(-\xi)\,|\xi|\ln{(|\xi|)}e^\xi =
\left\{
\begin{array}{l}
|\xi|\ln{(|\xi|)}e^{\xi},\;\;\xi<0,\\
\\
0,\;\;\xi>0,
\end{array}
\right.
$$
which can be verified to be in $H^s(\mathbb{R})$ for all $s<3/2$ by the 
expression of its Fourier transform:
$$
\widehat{T}_1(w)=\frac{\ri (\ri w-1)^2\left(\ln(w^2+1)-2\arctan(w)%
-2\ri (\gamma-1)\right)}{2(w^2+1)^2},
$$
where $\gamma$ is Euler's constant.
\end{proof}
\end{proof}

{
\begin{remark}
\label{Re}
In Theorem \ref{im}, we use the space $L^2(\mathbb{R})\cap C_{\rm{d}}(\mathbb{R})$ 
(with $C_{\rm{d}}(\mathbb{R})$ defined in \eqref{Cd}).
If more regularity is required by using the space 
\eqnn{H_{\rm{d}}(\mathbb{R})=
\left\{v\in L^2(\mathbb{R}) \Big| 
\;v |_{(-\infty,0)}\in H^1\left((-\infty,0)\right)\;%
{\rm{ and }}\;v |_{(0,\infty)}\in H^1\left((0,\infty)\right)\right\}
}
instead, one 
finds the point spectrum in the first part of Theorem \ref{im} to be 
$0< |\realpart{\lambda}|\leq c(3/2-b)$. Indeed, the proof of Theorem \ref{im} 
goes through with the modification that the condition $\realpart{r_3}\geq 1/2$ 
(instead of $\realpart{r_3}> 0$ specified in Eq.~\eqref{Rr3c}) must be satisfied 
in order for $F$ to be in  $H_{\rm{d}}(\mathbb{R})$. The closure of ${\mathcal{L}}_w$ holds because, 
for any interval $I\subset \mathbb{R}$,
$\|v\|_{L^\infty(I)}\leq C_s\| v\|_{H^1(I)}$ for some constant $C_s$. Furthermore, if $H^1$ is 
replaced by $H^s$, $1\leq s<3/2$, in the definition of $H_{\rm{d}}(\mathbb{R})$ above, then
the condition on $r_3$ becomes $\realpart{r_3}>  s-1/2$. This follows 
from the Fourier transform computation done in the proof of Lemma \ref{l35} 
and from Lemma 5.2 of \cite{Hitch} giving a criterion for a function to be in a fractional Sobolev space on a subset of $\mathbb{R}$. 
The band specified in the first part of Theorem \ref{im} is then 
found to be 
\eq{
0< |\realpart{\lambda}|< c(5/2-s-b),
}{Ul}
which limits to the band specified in the second part of Theorem \ref{im}  as $s\rightarrow 3/2$.
\end{remark}
}

\section{Numerical Results}
\label{numer_res}
In this section, we present numerical results concerning the existence and
spectral stability of standing and traveling wave solutions to the $b$-family 
of  equations, i.e., Eq.~\eqref{bfamily}. The discussion that follows next
is complemented by systematically presenting results on spatio-temporal 
evolution of generic (Gaussian) and peakon initial data.

\subsection{Standing and traveling waves}
{First, we shall be interested in the ``lefton'' solutions. 
A single lefton is a stationary solution of Eq.~(\ref{bfamily}) given by the 
explicit formula \cite{dhh2}}
\beq\label{lefton} 
u = A\,\Big(\cosh \gamma (x-x_0)\Big)^{-\frac{1}{\gamma}}, \qquad \gamma = -\frac{b+1}{2},  
\eeq 
where $A$ and $x_{0}$ are its amplitude and center, respectively. For a
given $b$, this is a 2-parameter family of solutions, given the arbitrary 
choice  $A$ and $x_0$. The form of Eq.~\eqref{lefton} suggests that leftons 
exist only for the parameter regime $b <-1$. This is confirmed numerically 
by parameter continuation in $b$. We start with a lefton solution $u$ given 
by Eq.~\eqref{lefton} with $b = -1.2$ and $x_0 = 0$, normalized so that 
$\| u \|_{L^2}^2=1$. We then increase $b$ using a secant-based predictor-corrector
parameter continuation algorithm in \textsc{Matlab}. So that a single member 
of the 2-parameter family is selected, we add the constraints that $u$ is an 
even function and that $\| u \|_{L^2}^2=1$. In all cases, the parameter continuation 
stops just before $b=-1$ is reached. Since the width of the lefton solution 
[cf. Eq.~\eqref{lefton}] increases as $b$ approaches $-1$, the exact stopping 
point depends on the domain size used for continuation, as well as the discretization 
of the problem (e.g.~number of grid points used) and the continuation step size.

We investigate the spectral stability of a lefton solution $u=u_0(x)$ of the 
$b$-family as written in Eq.~\eqref{bfamily}. We linearize the $b$-family about 
$u=u_0(x)$ and obtain the following eigenvalue problem
\beq 
\label{eigp}
\lambda(v-v'')+\left(c(v''-v)+(b+1)u_0v+(1-b)u_0'v'-u_0v''-u_0''v\right)'=0,
\eeq
where the prime denotes derivative with respect to $\xi$. We rearrange Eq.~\eqref{eigp} 
to get $(I - \partial_{\xi}^2)^{-1} \mathcal{L}(u_0) v = \lambda v$, where $\mathcal{L}(u_0)$ 
is the linear operator
\begin{equation}\notag
\mathcal{L}(u_0) = -\partial_{\xi}\left(c(\partial_{\xi}^2-I)+(b+1)u_0 I +(1-b)u_0' %
\partial_{\xi} -u_0 \partial_{\xi}^2 -u_0'' \right).
\end{equation} 
We can verify directly that $\mathcal{L}(u_0) u_0' = 0$, which results from translation
invariance of the system. In addition, when $c = 0$, $\mathcal{L}(u_0) u_0 = 0$. To find 
the spectrum, we {again} use Fourier spectral differentiation matrices for the differential 
operators and compute the eigenvalues using the built-in eigenvalue solver \texttt{eig} in 
\textsc{MATLAB}. Figure \ref{leftoneig} shows the computed spectrum for a lefton solution 
with parameter $b = -1.1$ and amplitude $A = 1$ ($c = 0$ for all leftons). The maximum real 
part of the spectrum is of order $10^{-7}$, suggesting that the spectrum is purely imaginary. 
In addition, we verify numerically that $(I - \partial_{\xi}^2)^{-1} \mathcal{L}(u_0) u_0' = 0$ 
and $(I - \partial_{\xi}^2)^{-1} \mathcal{L}(u_0) u_0 = 0$. We expect that the additional 
degree of freedom in $A$ in Eq.~\eqref{lefton} will generate an eigenfunction in the kernel of 
$(I - \partial_{\xi}^2)^{-1} \mathcal{L}(u_0)$, and we can verify numerically that 
$\partial u_0/\partial A = u_0$. The same spectral results are obtained for a wide range of 
$A$ and $b < -1$.

\begin{figure}[pt!]
\begin{center}
\includegraphics[height=.25\textheight]{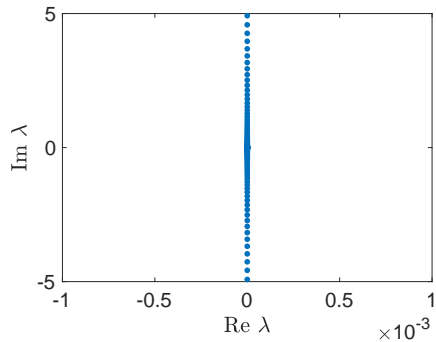}
\end{center}
\caption{Spectrum of $(I - \partial_x^2)^{-1} \mathcal{L}(u_0)$ for a lefton solution
with parameter values $A = 1$ and $b = -1.1$ which is obtained by using Fourier spectral 
methods with $N = 1024$ grid points (and periodic boundary conditions).
}
\label{leftoneig}
\end{figure}

For the peakon solutions, which are traveling waves, this method of computing the spectrum
does not work since the peakon is not differentiable at its center. As an alternative, we will 
compute the spectrum of the family of {\it smooth} solitary waves on a nonzero background \cite{Guo05}, 
which are solutions to the equation 
\begin{equation} \label{Guosolitoneq} 
c(u_{\xi\xi}-u) + (b+1)\frac{u^2}{2}+(1-b)\frac{u_{\xi}^2}{2}-uu_{\xi \xi} = g,\;\;\xi=x-ct,
\end{equation}
obtained by integrating the co-traveling frame ODE obtained from Eq.~\eqref{bfamily}.
The limit of these smooth solitons, which we compute numerically by parameter continuation 
in $g$ (Figure \ref{Guosolitons}, left panel), is the peakon
solution.

\begin{figure}[!pt]
\begin{center}
\includegraphics[height=.25\textheight]{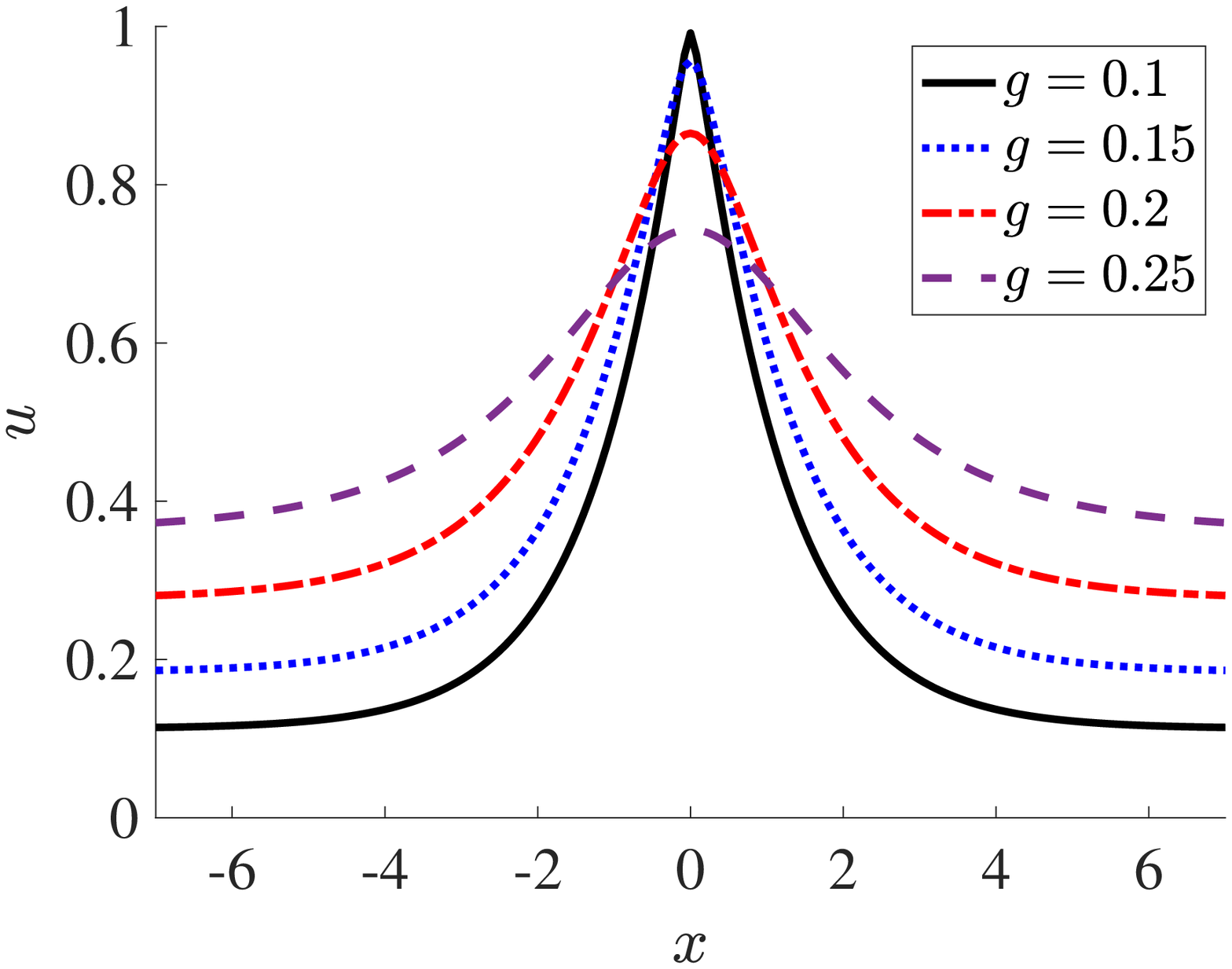}
\hskip -0.5cm
\includegraphics[height=.25\textheight]{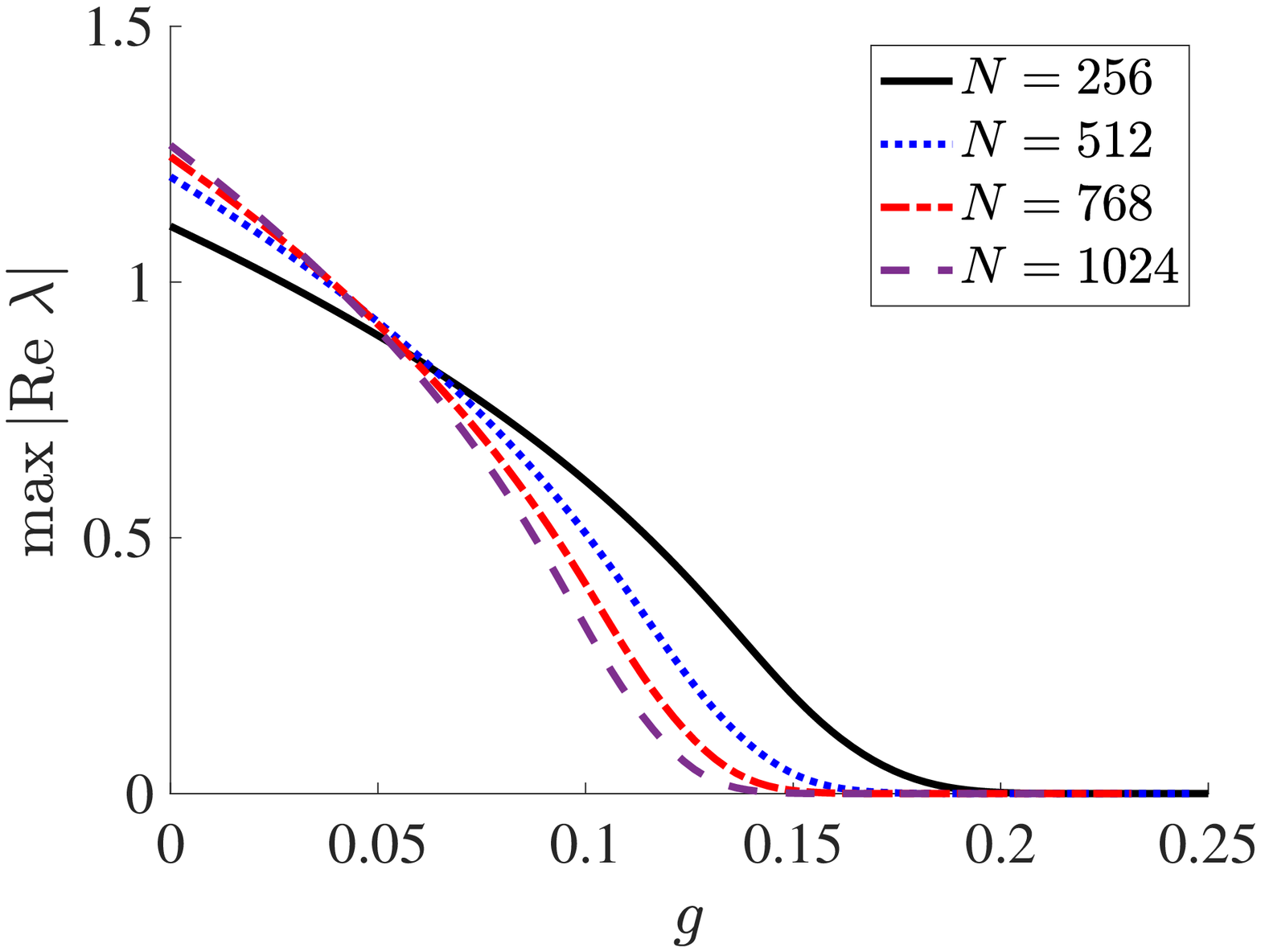} \\
\includegraphics[height=.25\textheight]{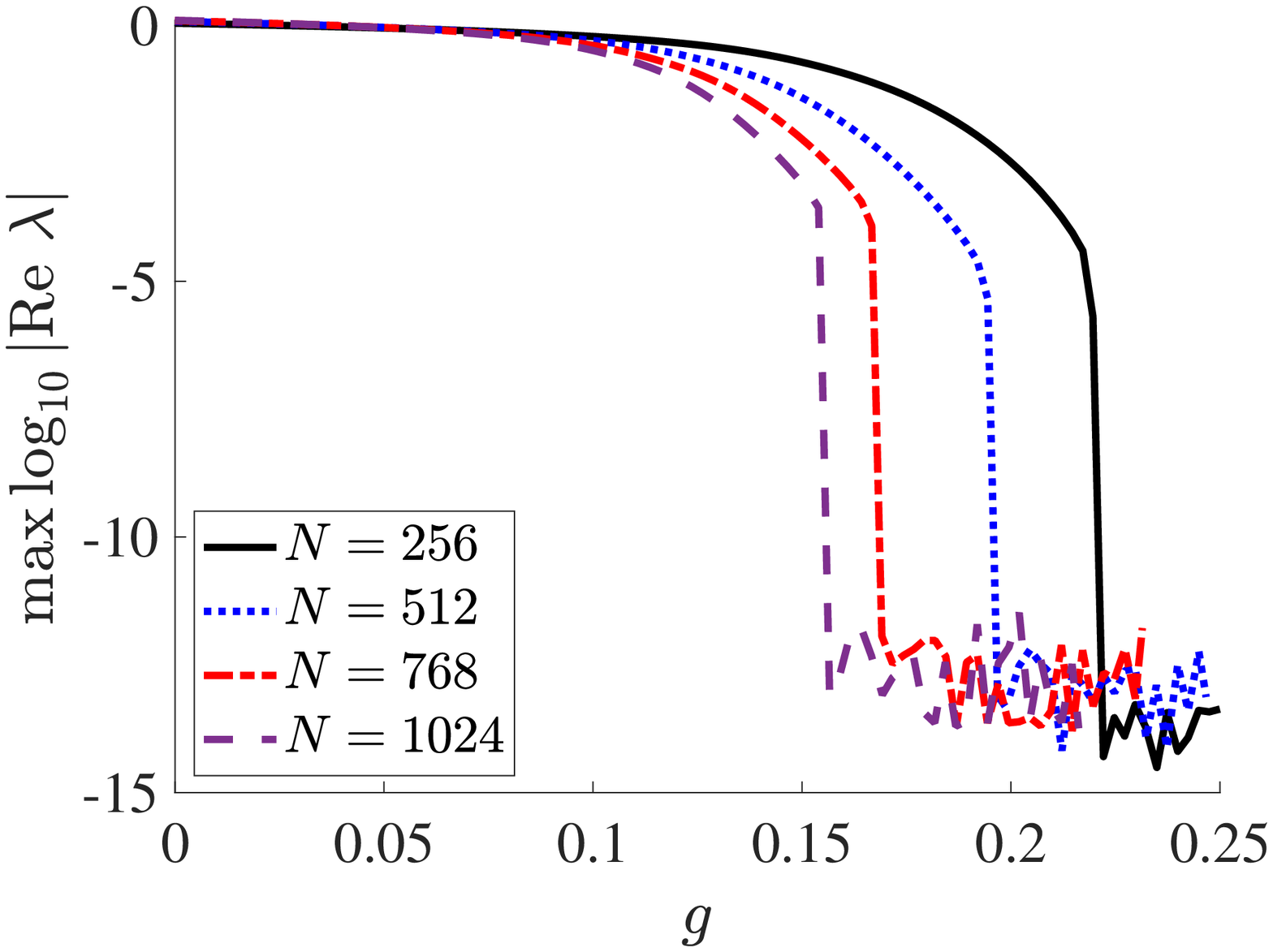} 
\hskip -0.5cm
\includegraphics[height=.25\textheight]{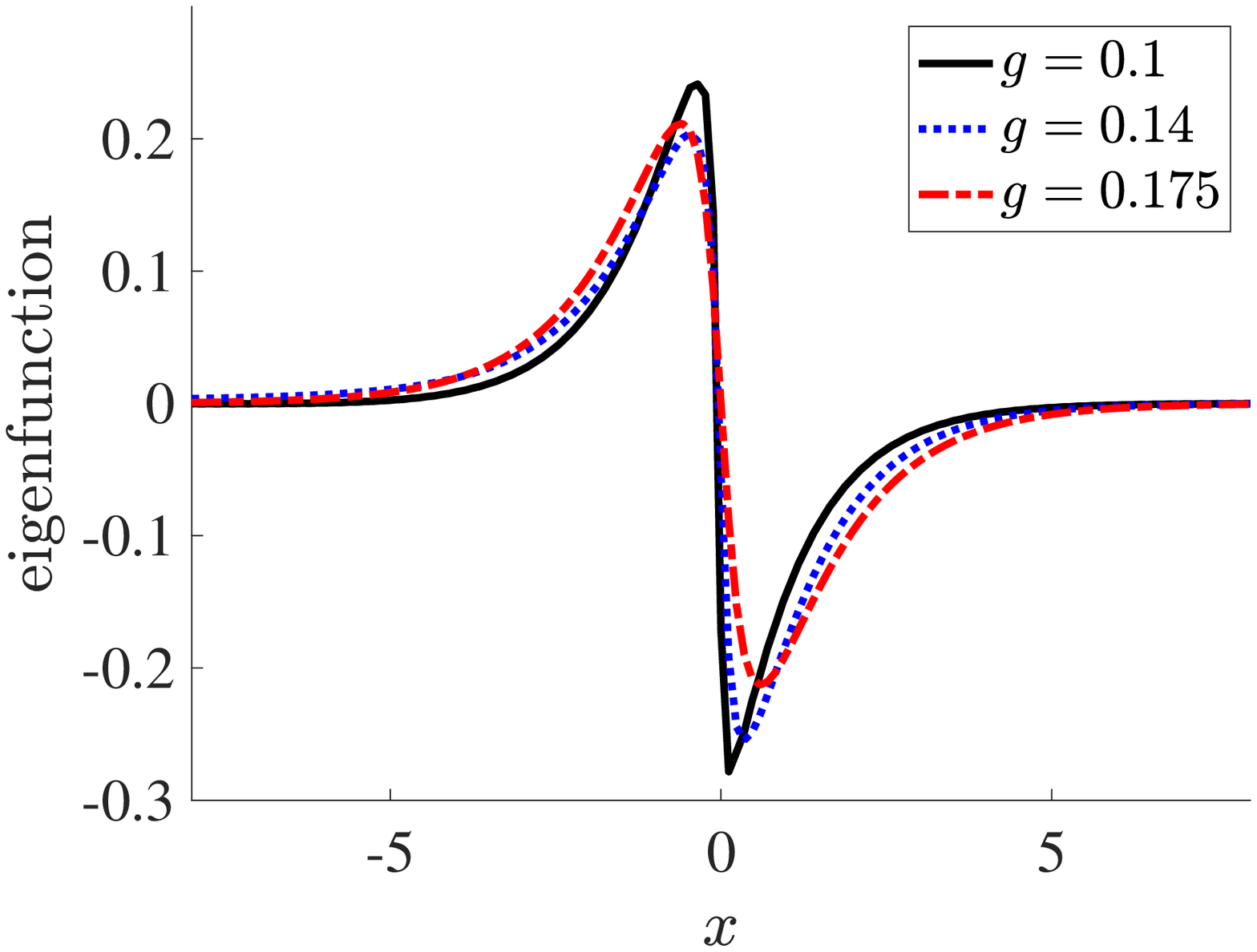} 
\end{center}
\caption{Smooth soliton solutions to Eq.~\eqref{Guosolitoneq}, for decreasing values 
of $g$ (top left). Maximum real part of spectrum versus parameter $g$ (top right) and log 
of maximum real part of spectrum versus the parameter $g$ (bottom left) for varying numbers 
$N$ of Fourier grid points. Eigenfunctions associated with the respective positive real 
eigenvalue for varying $g$ (bottom right). Similar to Fig.~\ref{leftoneig}, Fourier spectral 
methods and periodic boundary conditions have been used herein for $b = 1$.}
\label{Guosolitons}
\end{figure}

Using the same techniques as above, we can numerically compute the spectrum of these smooth
solitons. For $b = 1$, the maximum real part of the spectrum is of order $10^{-13}$ for 
sufficiently large $g$ (Figure \ref{Guosolitons}, bottom left), which suggests that the 
spectrum is purely imaginary for that parameter regime. When $g$ decreases below a threshold 
value, numerical spectral computation suggests the presence of an eigenvalue with positive 
real part. This threshold, however, is lower as the number of Fourier modes in the discretization 
is increased. Furthermore, the eigenfunction associated with this eigenvalue resembles the 
derivative of the smooth soliton (Figure \ref{Guosolitons}, bottom right), and becomes increasingly 
singular as $g$ decreases. Since the derivative is an eigenfunction with eigenvalue $0$ due to
translation invariance, this positive real eigenvalue is most likely an artifact resulting from 
the fact that the solution we are linearizing around becomes increasingly non-smooth as $g$ 
decreases, and this occurs sooner for coarser discretizations. Similar results are obtained 
for values of $b$ between $1$ and $2$. Thus, we conclude that for sufficiently large $g$ 
(i.e. $g>0.15$; cf.~Fig. \ref{Guosolitons}), the solutions of~\cite{Guo05} are spectrally stable; 
yet, as $g$ approaches $0$, we are no longer able to provide definitive spectral conclusions for 
the stability of the non-smooth peakon solutions, although the above interpretation of our spectral 
computations (corroborated by dynamical simulations given below) is suggestive of their robustness.

\subsection{Numerical timestepping}
We now turn our focus to spatio-temporal dynamics of the $b$-family of peakon
equations Eq.~\eqref{bfamily}. For our subsequent analysis, we will 
consider Gaussian initial data of the form of
\begin{equation}
u(x,t=0)\doteq \frac{1}{\sigma\sqrt{\pi}}e^{-\frac{\left(x-x_{0}\right)^2}{\sigma^{2}}},
\label{gaussian}
\end{equation}
where $\sigma$ and $x_{0}$ correspond to the width and center of the Gaussian
pulse, respectively. The previous works of~\cite{Holm03a,Holm03,FringHolm} considered 
the so-called $m$-formulation 
\begin{align}
m_{t}=-um_{x}-bu_{x}m, \quad m\doteq u-u_{xx},
\label{m_form}
\end{align}
which we adopt from now on, and the numerical scheme we employed in this work
is discussed next. We advance Eq.~\eqref{m_form} forward in time with the initial 
data of Eq.~\eqref{gaussian} by using Fourier spectral collocation for the spatial 
discretization supplemented by periodic boundary conditions on $[0,200]$, and the 
Runge-Kutta-Fehlberg (RKF45) for the time marching. The latter is a predictor-corrector 
method (with time step-size adaptation) where we used strict (absolute and relative) 
tolerances of $10^{-8}$ per time step. Then, at each time step, the field $u$ is obtained
from $m$ by inverting the Helmholtz operator $1-\partial_{x}^{2}$ in Fourier space. We 
should mention in passing that the time integration is performed in Fourier space as 
well. We remove the aliasing errors by employing the so-called $3/2$-rule in order to 
ensure that the high wavenumber Fourier coefficients are well decayed (see, e.g., Ref.~ \cite{jb2001}). 
However, we do not employ artificial viscosity as opposed to the works of~\cite{Holm03a,Holm03,FringHolm}. 
This way, it is expected that the numerical results reported herein are close representations 
of the original physical system.

A series of benchmarks of the numerical scheme is discussed in the Appendix~\ref{append_bench}.
In particular, using the initial data of Eq.~\eqref{gaussian}, selected cases of 
spatio-temporal dynamics in $b$ are presented giving rise to peakons, leftons as 
well as ramp-cliffs, and the results discussed therein are connected with the current 
literature. For example (see also Figs.~\ref{fig2}-\ref{fig4}), when $b<-1$, we observe 
the emergence of solitary pulses from Gaussian initial data [cf. Eq.~\eqref{gaussian}] 
that move to the left, gradually asymptoting to a steady-state solution, i.e., leftons 
[cf. Eq.~\eqref{lefton}]. It should be noted in passing that the number of leftons depends 
on how close or far away the selected value of $b$ is from $-1$, e.g., we observed the emergence 
of three and two leftons for $b=-2$ and $b=-1.5$, respectively (see the Appendix~\ref{append_bench} 
for a detailed discussion on leftons).

\begin{figure}[pt!]
\begin{center}
\includegraphics[height=.23\textheight, angle =0]{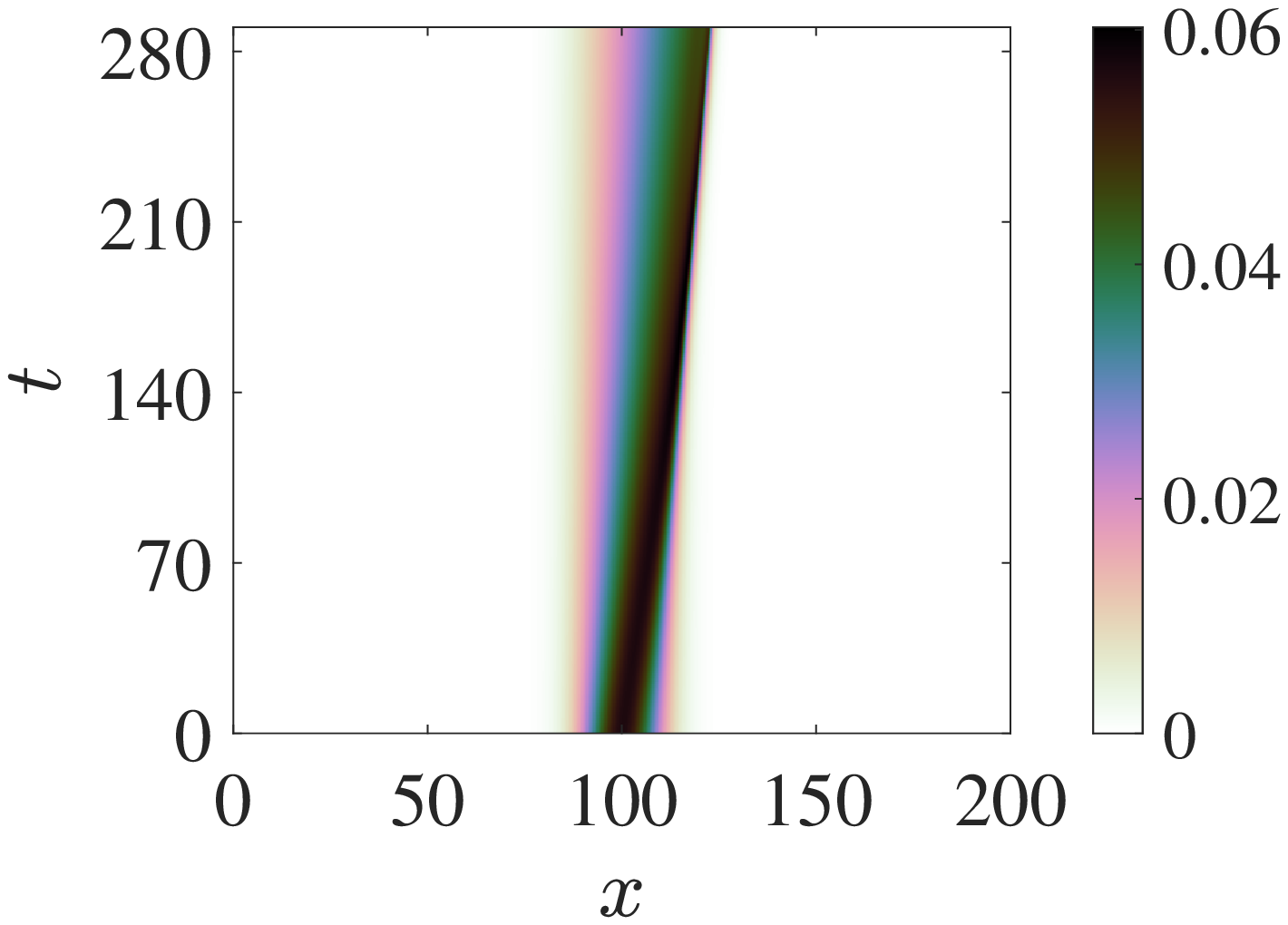}
\includegraphics[height=.23\textheight, angle =0]{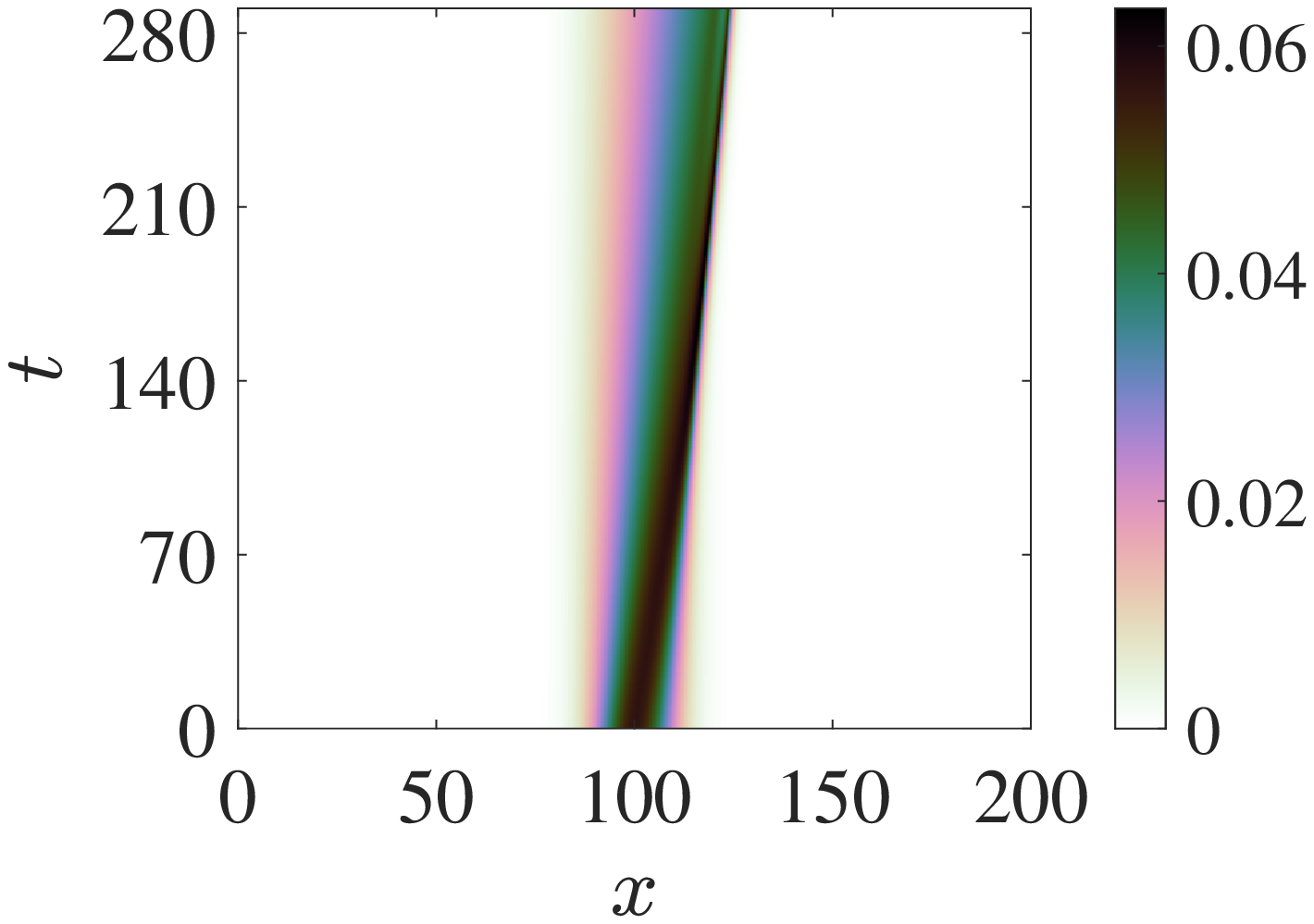}
\end{center}
\caption{Contour plots of spatio-temporal evolution of ramp-cliff solutions
generated by Gaussian initial data with $\sigma=10$ and $x_{0}=100$, and
$N=32768$ Fourier modes. The left and right panels correspond to $b=0.8$ 
and $b=0.99$, respectively.
}
\label{fig_RC}
\end{figure}

We now turn our focus on the ramp-cliff regime corresponding to the case 
when $b\in (-1,1)$. In the Appendix~\ref{append_bench}, we present 4 cases 
of ramp-cliffs where the latter travel faster for gradually increasing values 
of $b$. We can clearly observe the formation of these patterns and the self-similar 
expansion of their rear tails, while at the same time their front part steepens. 
It is worth noting here that we are not aware of a frame where such solutions 
can be considered as steady. However, we report at this point an artifact that 
was observed in our numerical simulations with $N=16384$ collocation points and 
interval of time of integration $t\in[0,3000]$. One would expect the emergence of 
ramp-cliffs propagating to the right of the computational domain. Nevertheless, 
for $b\gtrapprox 0.85$ we noticed that peakons were emitted from the ramp-cliffs, 
with the former emerging as robust traveling waves. We investigated this byproduct 
of the numerical scheme by considering the implications of Theorem 3 in~\cite{Hone14}. 
In particular, it can be shown that if $m(x,t=0)>0$, then $m(x,t)>0$, $\forall t>0$ 
holds which in fact is the case as per the Gaussian initial data employed in this work. 

\begin{figure}[!pt]
\begin{center}
\includegraphics[height=.23\textheight, angle =0]{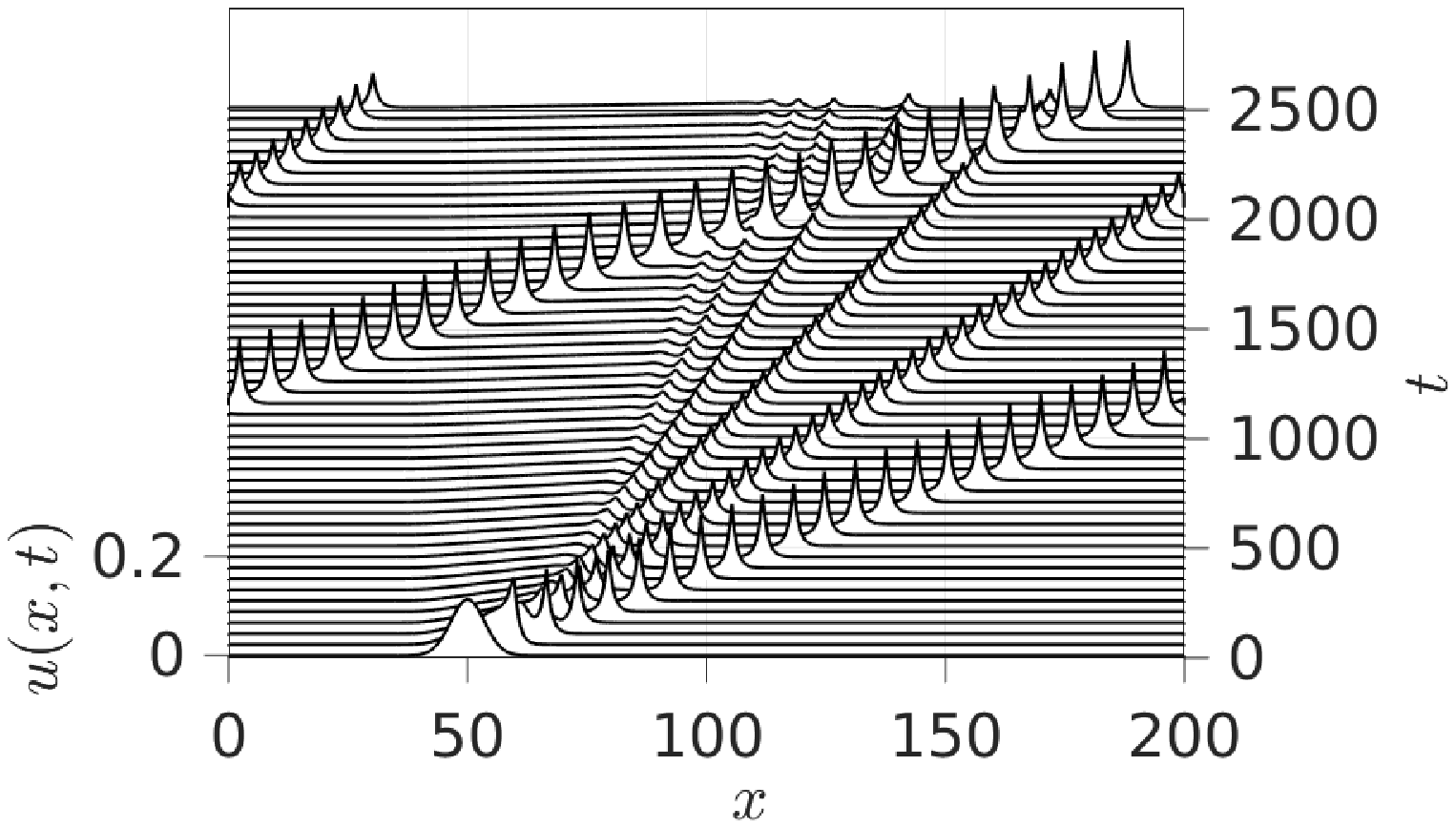}
\hskip -0.5cm
\includegraphics[height=.23\textheight, angle =0]{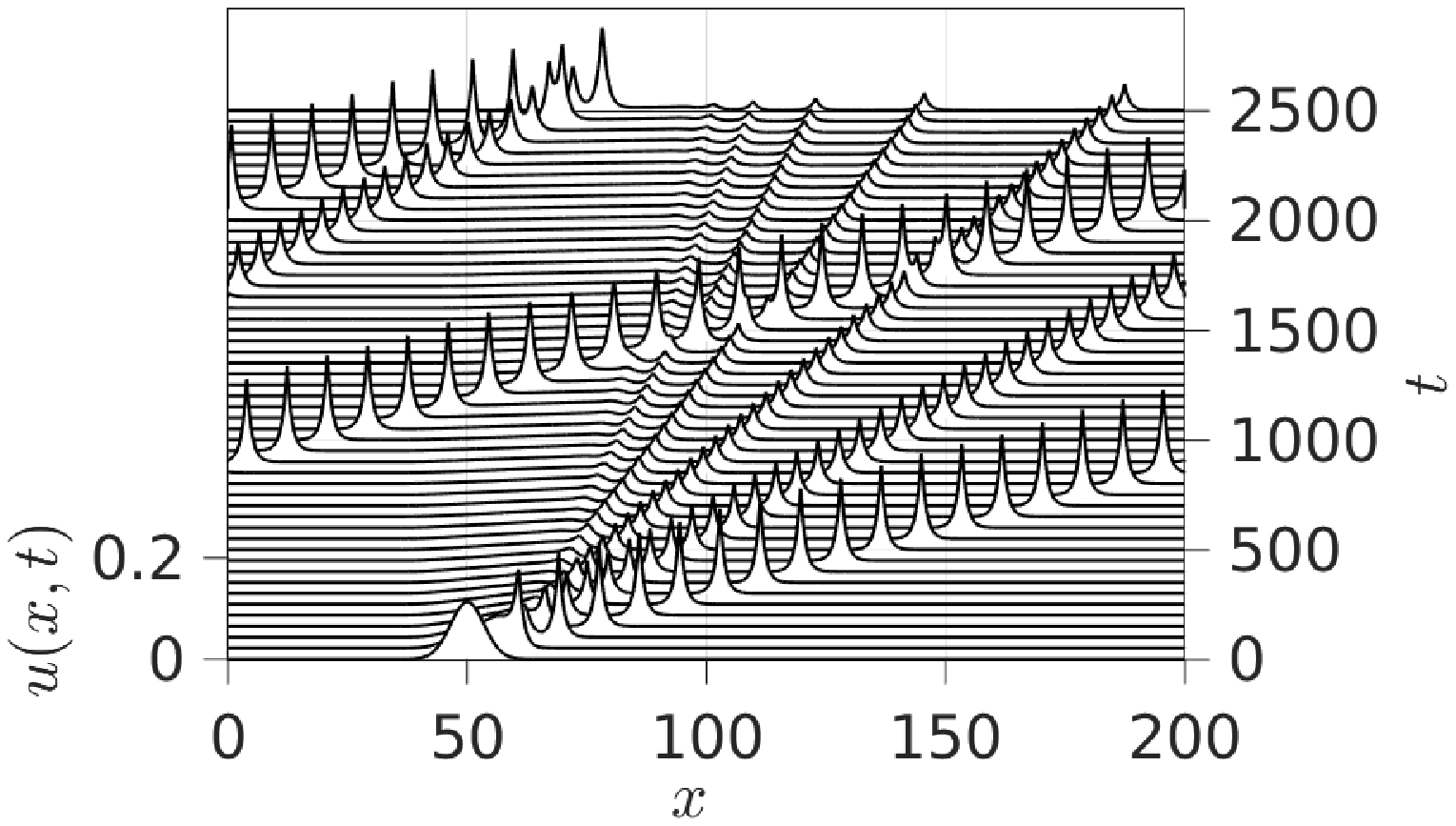}\\
\includegraphics[height=.23\textheight, angle =0]{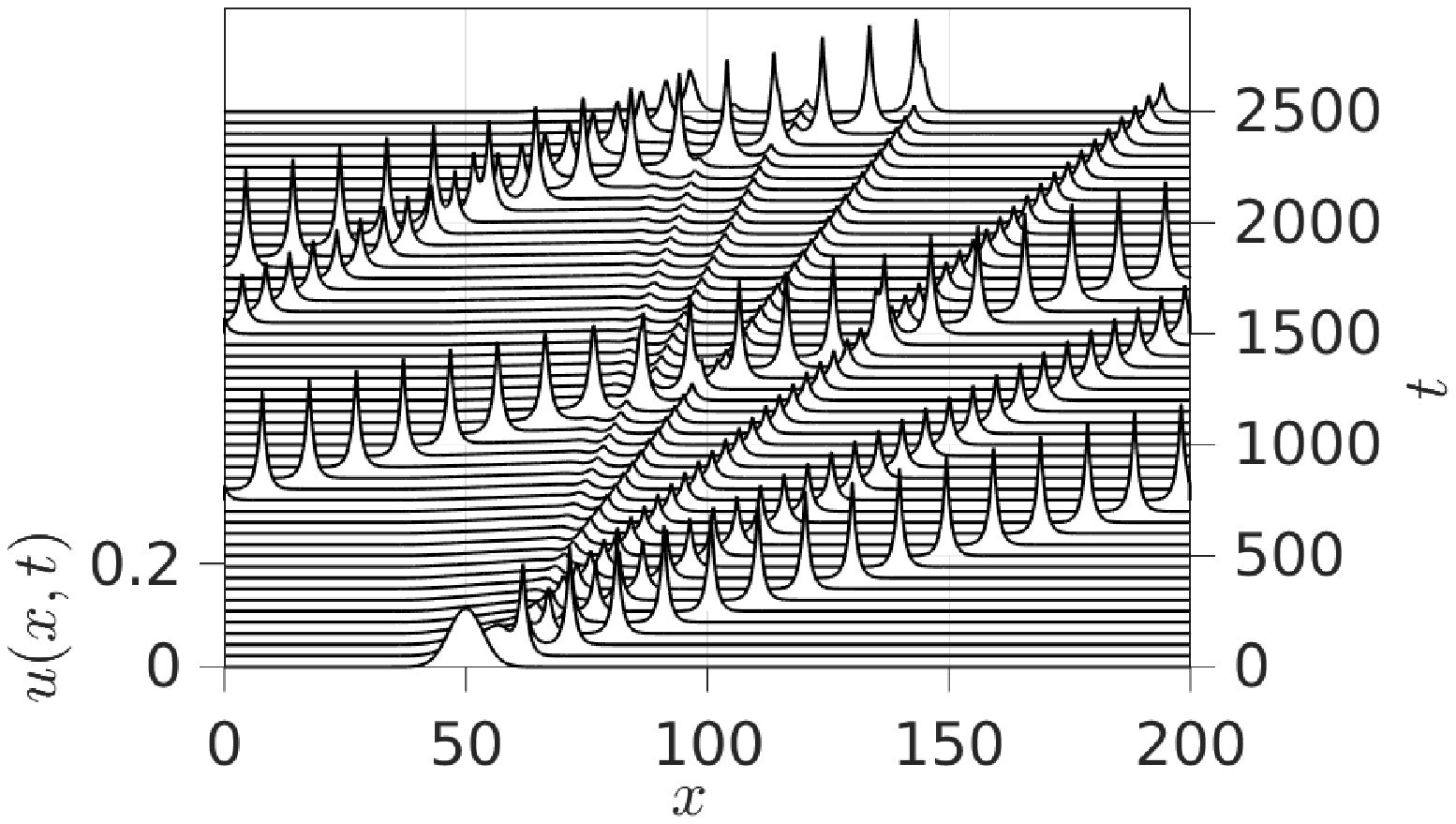}
\hskip -0.5cm
\includegraphics[height=.23\textheight, angle =0]{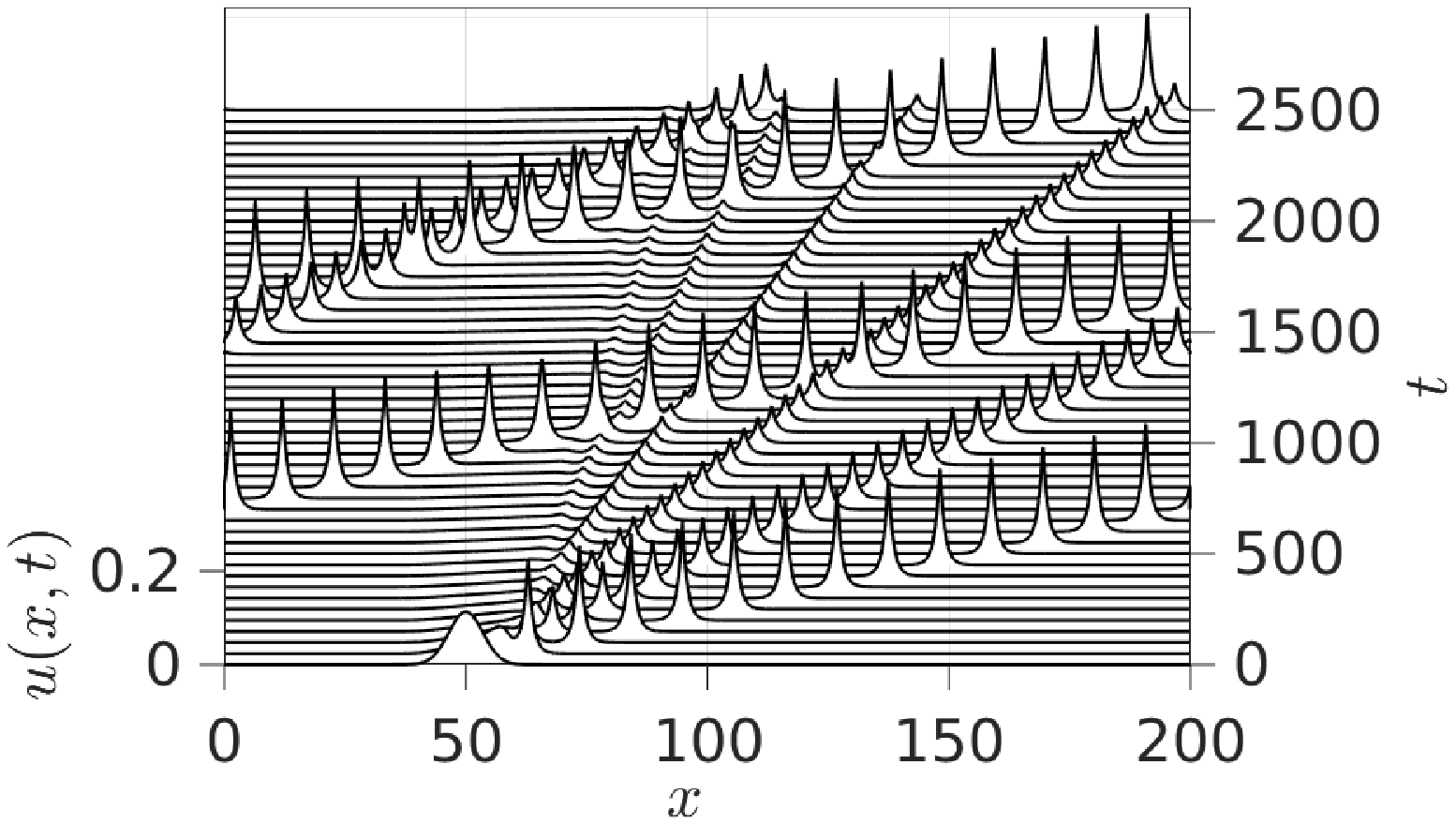}\\
\end{center}
\caption{
Same as Fig.~\ref{fig3} but for Gaussian initial data with $\sigma=5$ (and
$x_{0}=50$) and $N=8192$ Fourier modes. Top left and right panels correspond 
to values of $b$ of $b=1.5$ and $b=2$ (CH) whereas the bottom left and right 
ones to values  of $b=2.5$ and $b=3$ (DP), respectively.
}
\label{fig6}
\end{figure}

Upon a careful inspection of the temporal evolution of the variable $m$, we 
noticed that it becomes negative past a time $t_{0}$, thus suggesting that 
one cannot continue the temporal integration beyond that time (due to the
numerical scheme violating a theoretically established constraint). Moreover, 
we performed a spatial grid refinement by increasing the number of collocation 
points to $N=32768$ in order to investigate further the dependence of $t_{0}$ 
on $N$. We still observed the emergence of such ``spurious'' peakons but their 
appearance was delayed in time. This finding is somewhat expected: in this computation, 
we keep our spatial domain $[0,200]$ fixed during the spatial grid refinement 
which implies that the wavenumbers are still multiples of $k=2\pi /L$. Thus, when 
the number of collocation points is increased, the numerical scheme resolves 
progressively better the large wavenumbers which, in turn, results in the time 
delay of the emergence of those ``spurious'' peakons. It is expected that if we 
increase the number of nodes to, e.g., $N=65536$, this artifact will gradually disappear. 
As case examples of ramp-cliffs (in addition to the ones shown in Fig.~\ref{fig5} in 
Appendix~\ref{append_bench}), we demonstrate two cases with $b=0.8$ and $b=0.99$ 
in Fig.~\ref{fig_RC} where we stopped the integrator at $t \approx 290$ (past that
time, we observed the non-positivity of the $m$ variable).

We now investigate the peakon regime of the $b$-family, i.e.,when $b>1$. In 
particular, Fig.~\ref{fig6} presents selective cases of numerical simulations 
based on Gaussian initial data with $\sigma=5$ and $x_{0}=50$, and $N=8192$ 
Fourier modes. The top left and right panels correspond to the cases with $b=1.5$ 
and $b=2$ (CH) whereas the bottom left and right to values of $b$ of $b=2.5$ and 
$b=3$ (DP), respectively. The emergence of sharply peaked waves can be discerned 
from these panels where the initial Gaussian pulse breaks into peakons as time 
progresses. Furthermore, the time when the first peakon emerges in the simulations 
depends on the value of $b$, that is, its emergence is ``delayed'' when $b$ is close
to $1$. However, when the value $b$ is further away from that limit, the first peakon 
emerges at earlier times together with secondary peakons of smaller amplitude traveling
across the computational grid. It should be noted also that the first peakon (having 
actually the largest amplitude) travels in the computational grid and undergoes nearly 
elastic collisions with other peakons of smaller amplitude. Such phenomenology is interesting 
in its own right and deserves further study, however it is beyond the scope of the present 
work.

We finally focus on Theorem~\ref{im} (see Section~\ref{main_res}) which suggests 
that the point spectrum contains positive eigenvalues for $b<1$, that is, the 
peakons are orbitally unstable for $b<1$. We explore this theoretical finding 
numerically by considering a peakon centered at $x_{0}=50$ with speed (or amplitude) 
$c\approx 0.031$, and $N=32768$ collocation points. The left and right panels of 
the top row of Figure~\ref{peakon_stab} present our numerical results for values 
of $b$ of $b=0.88$ (left panel) and $b=0.98$ (right panel), respectively. It can be 
discerned from both panels that the peakons are orbitally unstable. The amplitude of 
the initial profile ($t_{0}=0$) gradually increases over time eventually leading to 
a collapse of the waveform (in particular, past $t_{0}\approx 130$ for the spatial 
discretization employed herein). 

On the other hand, i.e., when $b>1$, we expect peakons to be orbitally stable.
Indeed, this is the case as is shown in the middle and bottom panels of Fig.~\ref{peakon_stab}.
In particular, the middle and bottom panels showcase profiles of peakons at 
$t_{0}=0$ and $t_{0}=3000$ (terminal time of integration) for $b=1.3$ and $b=1.5$,
respectively (the same initial condition was used in both cases as in the top 
row of Fig.~\ref{peakon_stab}). It can be discerned from both panels that peakons 
appear to be robust over the time integration. However, a couple of remarks are in 
order at this point and in line with the middle and bottom panels of Fig.~\ref{peakon_stab}. 
We observe a small in-amplitude yet stationary localized error at the vicinity
of the center ($x_{0}=50$) of the initially placed peakon. It has been argued 
in~\cite{mitso2019} that when non-smooth initial data are considered in an 
evolution numerical experiment (such as peakons in the $b$-family), localized 
errors are expected to be formed in the vicinity of $x_{0}$ initially that remain 
stationary in time. This is the case in both panels of Fig.~\ref{peakon_stab}
and it is expected that this error gradually diminishes with grid refinement 
(see~\cite{mitso2019}). However, this error results in a slightly larger amplitude 
(and thus speed) of the pertinent peakon waveform but after a ``transient'' 
period of time it remains constant over the time evolution, as this can be seen
in the insets of the panels. Indicatively, the location of the peakon after $3000$ 
time units in the bottom panel (i.e., for $b=1.5$) is found to be at $x\approx 145.4$ 
whereas the theoretical expectation is $\approx 143.1$, thus suggesting a (relative) 
error of $\approx 1.6\%$. Despite this artifact, peakons for $b>1$ appear to be highly
robust and these findings are in accordance with Theorem~\ref{im}.

\begin{figure}[pt!]
\begin{center}
\includegraphics[height=.23\textheight, angle =0]{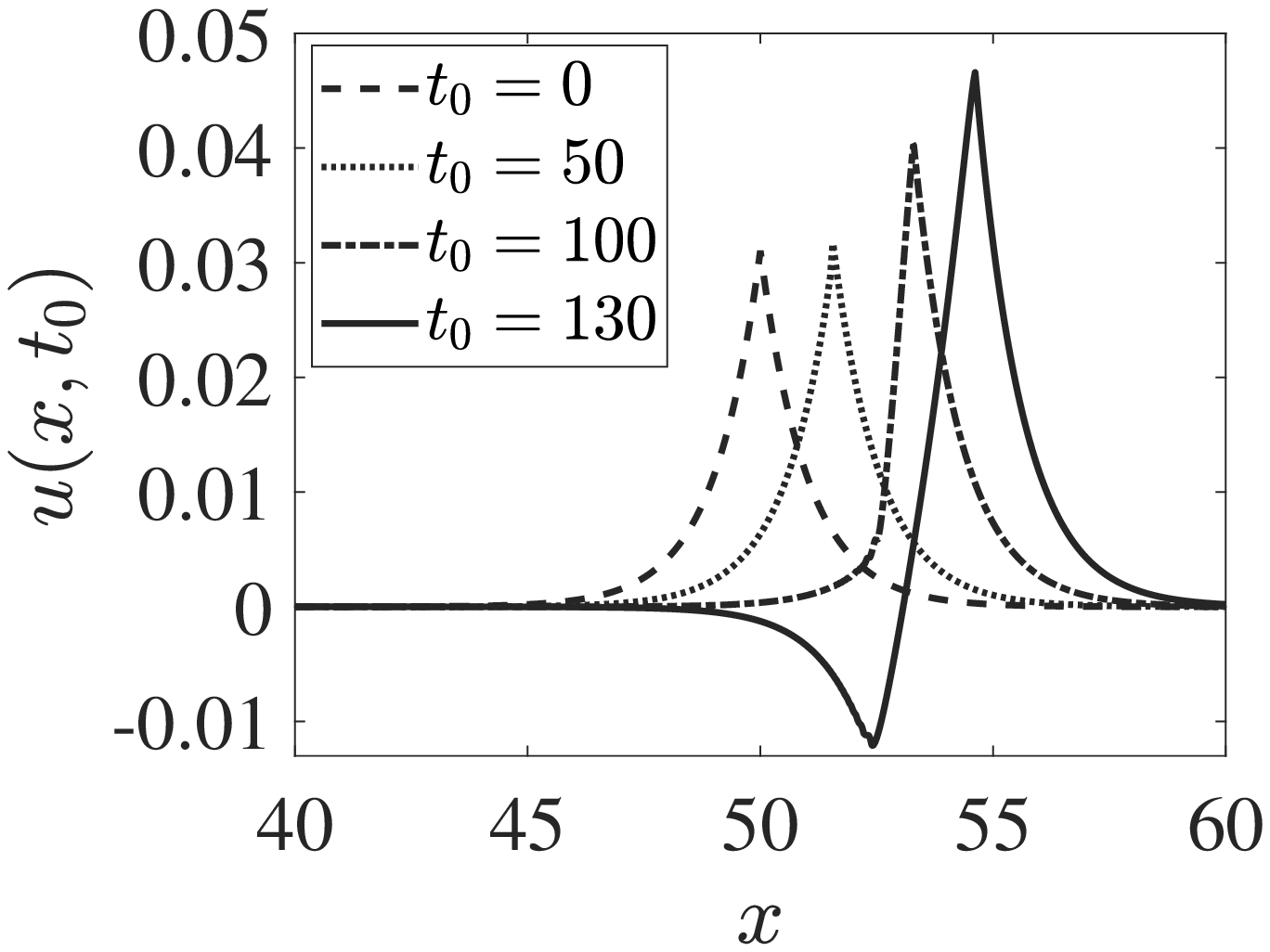}
\hskip -0.5cm
\includegraphics[height=.23\textheight, angle =0]{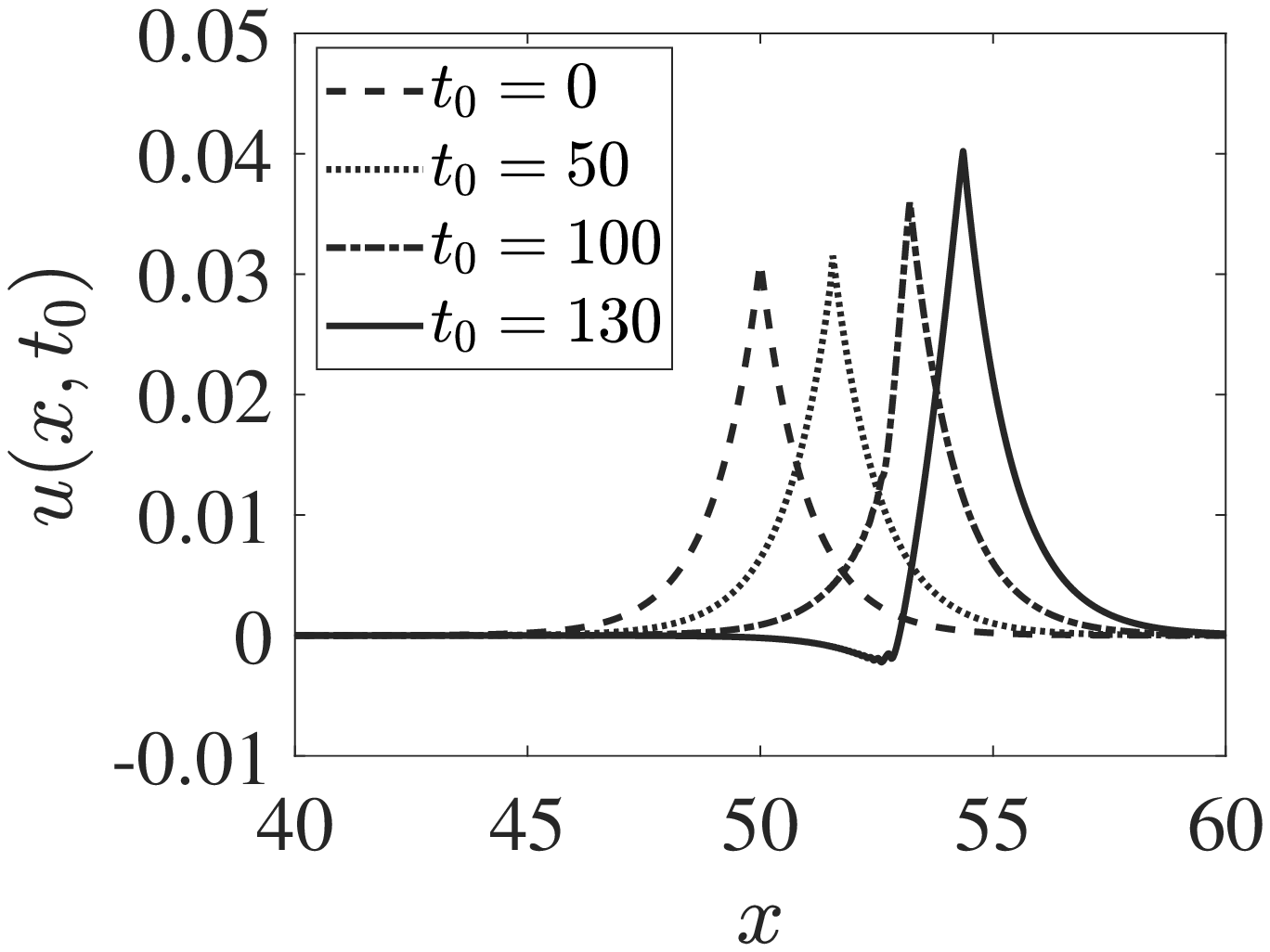}\\
\hskip 0.3cm
\includegraphics[height=.17\textheight, angle =0]{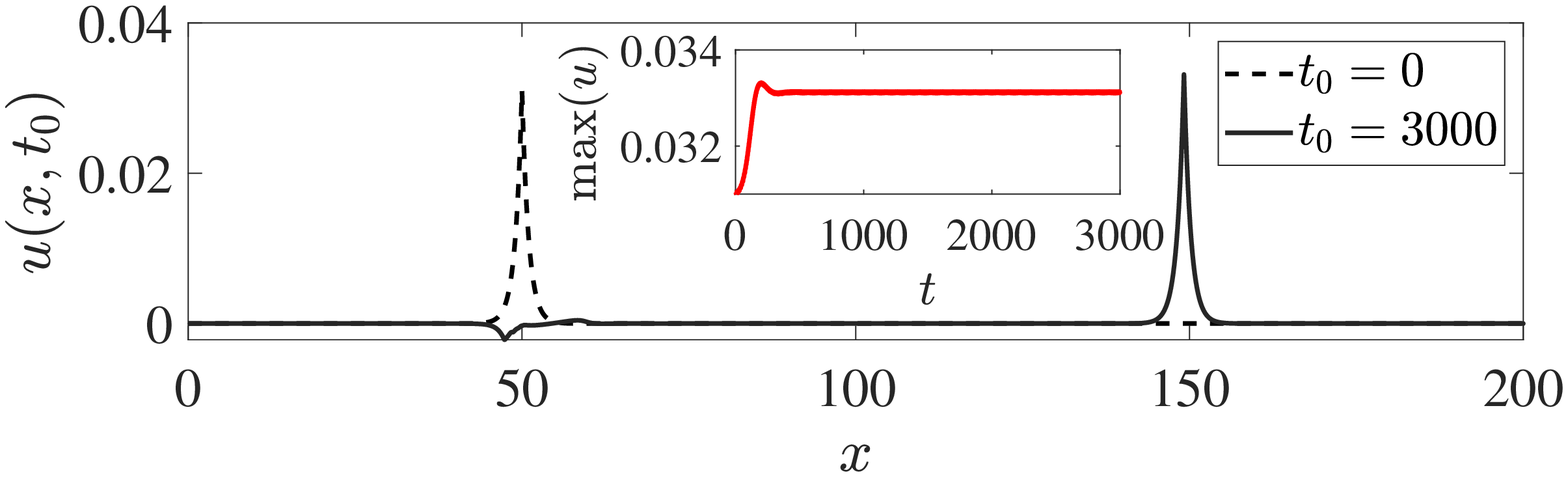}\\
\hskip 0.3cm
\includegraphics[height=.17\textheight, angle =0]{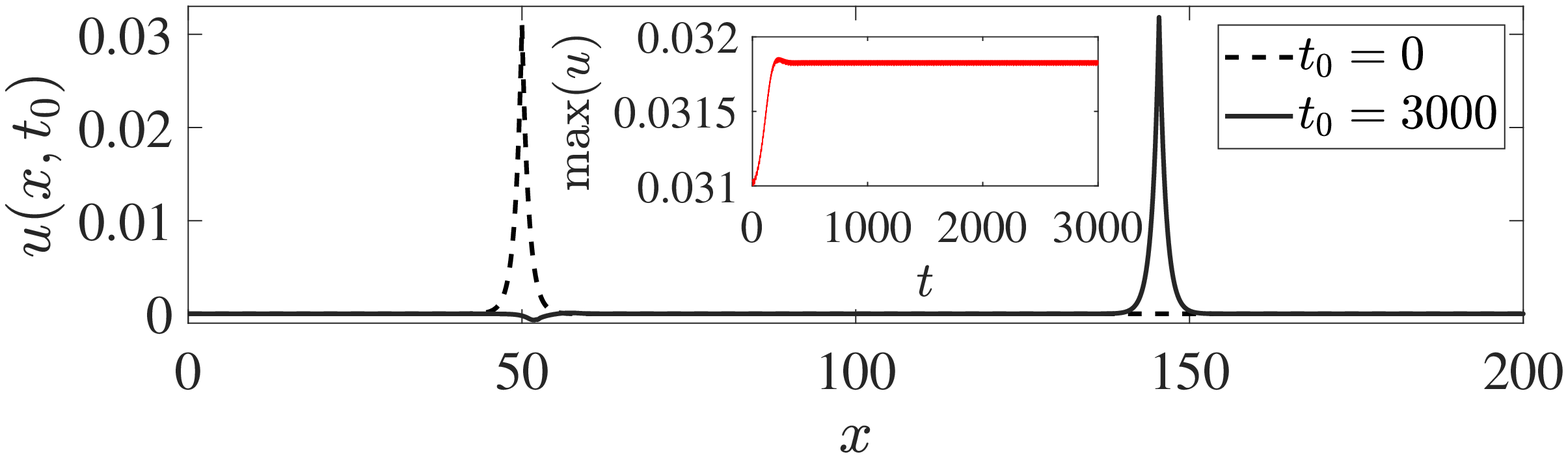}
\end{center}
\caption{
\textit{Top row}: The emergence of the instability for peakon solutions to the
$b$-family. Left and right panels present snapshots of peakon solutions at various
times $t_{0}$ (see the legend therein) for $b=0.88$ and $b=0.98$, respectively. 
\textit{Middle and bottom rows}: The stable regime $b>1$ for $b=1.3$ and $b=1.5$. 
The panel in the middle and bottom rows showcases a peakon solution at $t_{0}=0$ 
and $t_{0}=3000$ with dashed and solid black lines, respectively. The insets therein, 
demonstrate the amplitude of the peakon as a function of time with a solid red line. 
A peakon centered at $x_{0}=50$ with speed $c\approx 0.031$ is employed as an initial 
condition whereas the number of collocation points in all these panels is $N=32768$.
}
\label{peakon_stab}
\end{figure}

\section{Conclusions and Future Directions}
In the present work we have identified the solutions of the $b$-family
of peakon equations. We have provided some analytical insight on the
spectral problem, identifying the instability of the peakon waveforms 
via the consideration of their point spectrum. Indeed, we have indicated
that the latter contains eigenvalues with a positive real part. Our 
analytical insights have been corroborated by a diverse array of numerical 
computations. For structures that we could identify as steady, either in 
the original frame or in a co-traveling frame, we attempted to offer a 
complementary spectral picture. This was done in the case of the leftons 
for $b<-1$ which are stationary and were found to potentially be stable in 
this regime. On the other hand, in the regime $-1<b<1$, we could only perform 
dynamical simulations which illustrated the transient emergence and tendency 
towards breaking of ramp-cliff waveforms. The resulting formation of peakon 
structures (as $b \rightarrow 1$) was identified as a feature that disappears 
as the high wavenumbers become better resolved. However, the peakon structures 
become indeed dominant for $b>1$ where they spontaneously arise from smooth
initial conditions and robustly persist for different values of $b$, for 
integrable and non-integrable cases alike. Suggestive, although not definitive, 
towards their stability is the picture identified spectrally for the solutions 
on a finite background, tending towards these peakons as the background parameter 
$g$ tends to $0$.

While we believe that this study addresses some of the pending questions on this 
class of systems admittedly many more questions remain open and are worthwhile 
to explore in future studies. Is there a meaningful (and consistent with our theoretical 
analysis) way in which the peakon spectral analysis can be numerically performed? 
Is there a frame (possibly a self-similarly evolving one) where the ramp-cliff 
structures can be considered as steady and thus be spectrally analyzed? Are there
higher-dimensional analogues of these different structures and, if so, which of the 
above properties persist or disappear even in the two-spatial-dimension case?
These are only some among the numerous open questions. Work in these is currently 
underway and will be reported in future publications.


\section*{{\small Acknowledgments}}
PGK acknowledges support from the U.S.~National Science Foundation under
Grants no.~PHY-1602994 and DMS-1809074 (PGK). EGC is indebted to Hans 
Johnston (UMass) for endless support, discussions and guidance throughout 
this work. He thanks Darryl Holm (Imperial College) for pointing out 
Ref.~\cite{FringHolm} and express his gratitude to James (Mac) Hyman 
(Tulane University) for fruitful discussions during his visit at Los 
Alamos National Laboratory in 2019. He also express his gratitude to
Chi-Wang Shu (Brown University) for discussions about discontinuous 
Galerkin methods. SL acknowledges a Collaboration Grants for Mathematicians
from the Simons Foundation (award \# 420847).  SL also acknowledges discussions 
with Andrew Hone (University of Kent) and Simon Eveson (University of York).

\appendix

\section{Spatio-temporal dynamics: From peakons to Leftons and 
Ramp-Cliffs}
\label{append_bench}
We test our numerical scheme by re-producing a subset of the results of 
Refs.~\cite{dhh} and~\cite{Holm03a}. In particular, the left ($b=2$) and 
right ($b=3$) panels of Fig.~\ref{fig1} correspond to the spatio-temporal 
evolution of $u(x,t)$ by using Gaussian initial data [cf. Eq.~\eqref{gaussian}] 
with $\sigma=5$ and $x_{0}=100$, and $\sigma=5$ and $x_{0}=33$ respectively. 
Those results compare well with Figs 1 and 2 of Refs.~\cite{Holm03a} and~\cite{dhh},
respectively.

\begin{figure}[pt!]
\begin{center}
\includegraphics[height=.231\textheight, angle =0]{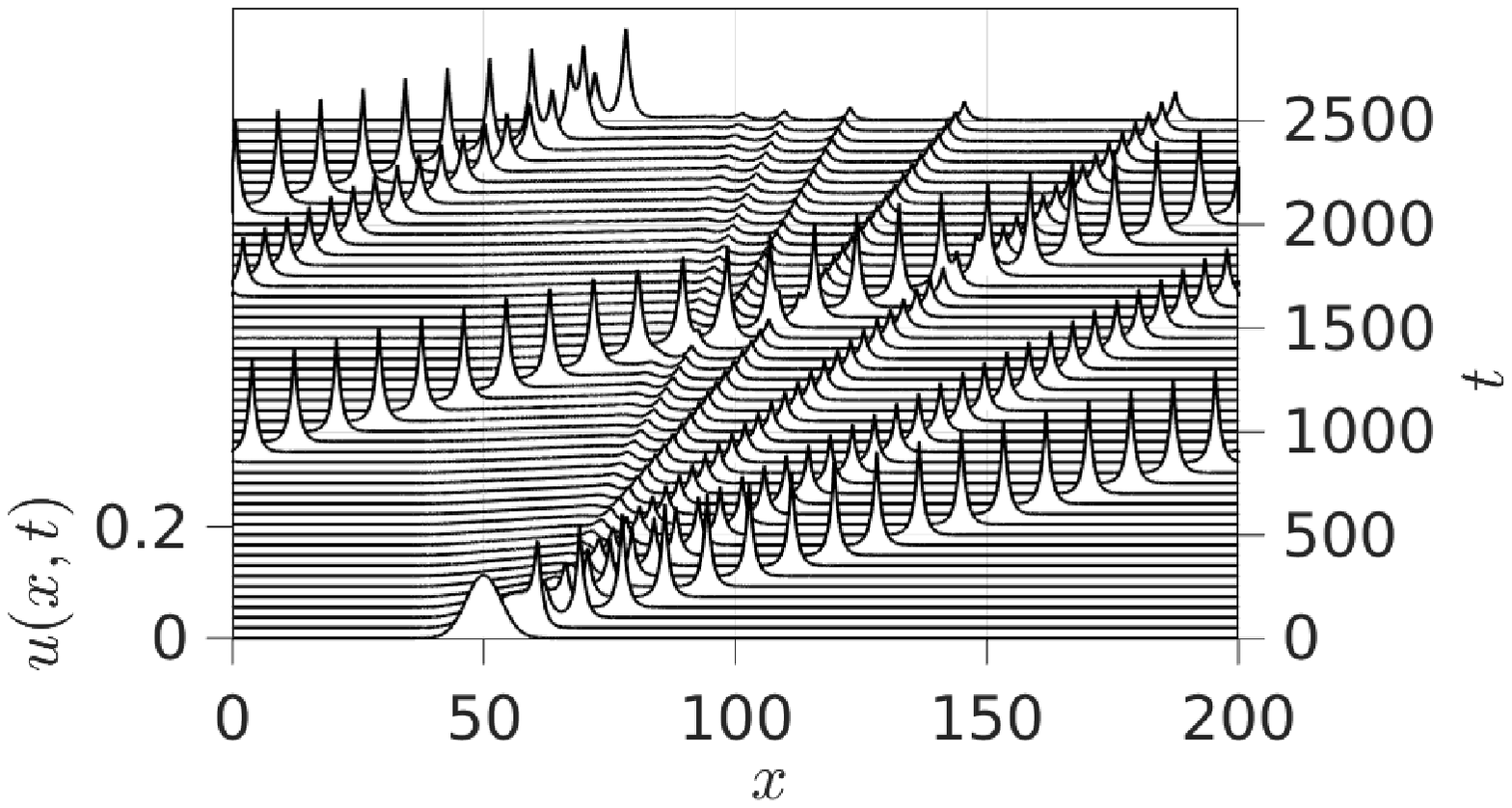}
\hskip -0.5cm
\includegraphics[height=.23\textheight, angle =0]{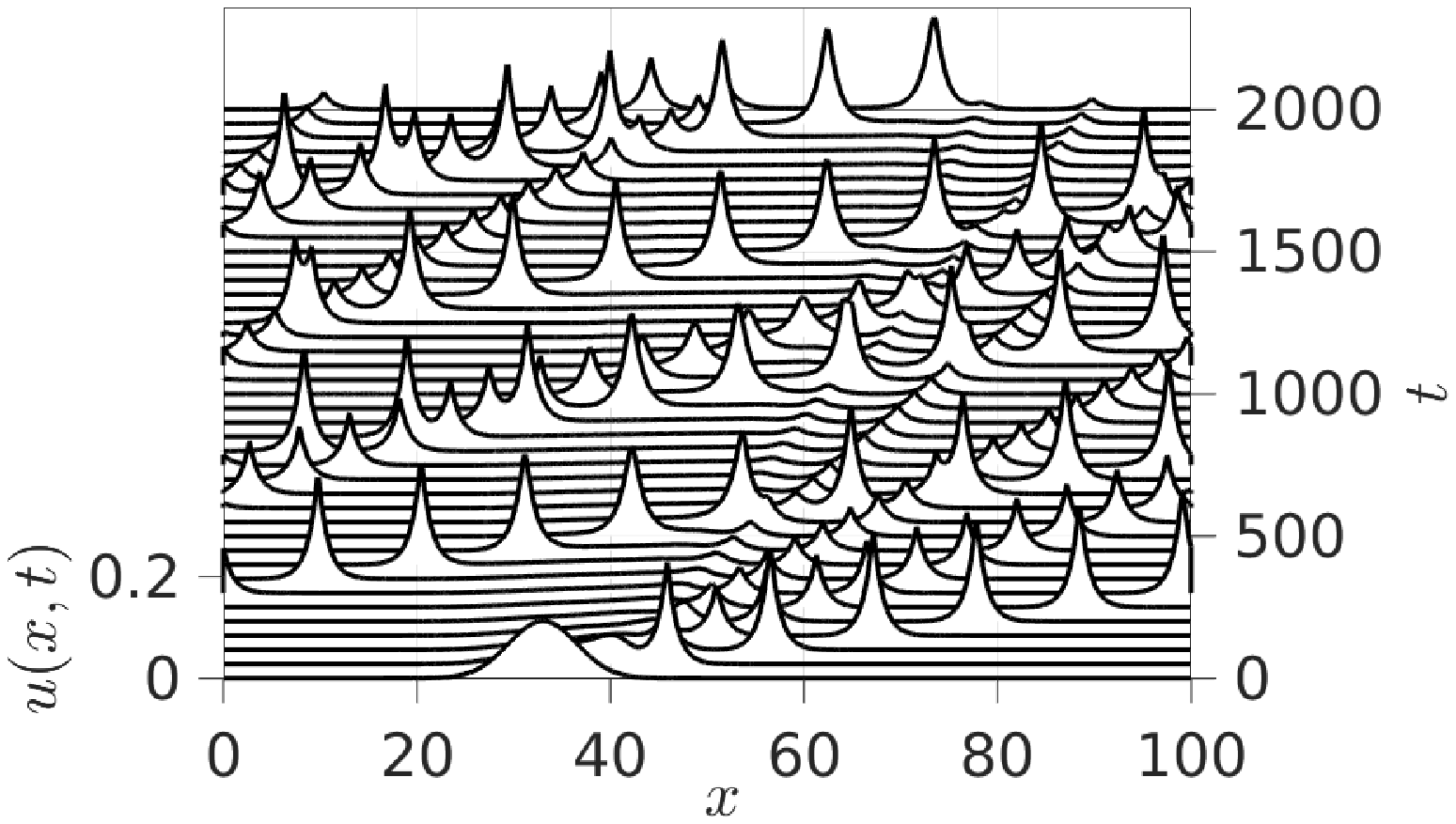}
\end{center}
\caption{
\textit{Left panel:}
Spatio-temporal evolution of a Gaussian profile with $\sigma=5$ 
centered at $x_{0}=50$ and $b=2$. Note that $N=8192$ Fourier
collocation points in space were used for this computation.
\textit{Right panel:}
Same as the left one but for Gaussian initial data with $x_{0}=33$
and $b=3$ (and same width, i.e., $\sigma=5$). Here, $N=4096$ 
collocation points were used.
}
\label{fig1}
\end{figure}
Next, we focus on the regime $b<-1$. In particular, Figs.~\ref{fig2}-\ref{fig4} 
highlight numerical results on the \textit{lefton} regime [cf. Eq.~\eqref{lefton}] 
by considering various values of $b$ (with $N=8192$ Fourier modes). In particular, 
Fig.~\ref{fig2} presents the spatio-temporal evolution of $u(x,t)$ for the cases 
with $b=-3$ (top left panel), $b=-2.5$ (top right panel), $b=-2$ (bottom left panel), 
and $b=-1.5$ (bottom right panel), respectively, when Gaussian initial data are 
employed with $\sigma=10$ and $x_{0}=100$. The emergence of leftons is clearly 
evident in all those panels and we notice the appearance of more leftons when 
$b (<-1)$ is larger in its absolute value (notice the appearance of four leftons 
in the top left and right panels whereas the bottom left and right ones contain 
three and two, respectively). We further investigated the emergence of leftons by 
considering different values of the Gaussian's width and center. Specifically, 
Fig.~\ref{fig3} presents results with $\sigma=7$ (and $x_{0}=100$) where the number 
of leftons decreases as $b$ approaches $-1$. 

Fig.~\ref{fig4} compares the numerically obtained (stationary) solution of the top 
right panel of Fig.~\ref{fig3} with Eq.~\eqref{lefton}. It should be noted that this
result is the analogue of Fig. 6 in~\cite{Holm03a}. In the present case (with $b=-2.5$), 
three leftons appear at the terminal time of the evolution ($t=2500$) whose locations 
and amplitudes are computed. Then, those values are plugged into Eq.~\eqref{lefton} and
are plotted with stars, crosses and plus signs in Fig.~\ref{fig4}. A perfect match can 
be clearly discerned, thus suggesting the accuracy and high-fidelity of the numerical 
scheme employed in this work.

\begin{figure}[pt!]
\begin{center}
\includegraphics[height=.23\textheight, angle =0]{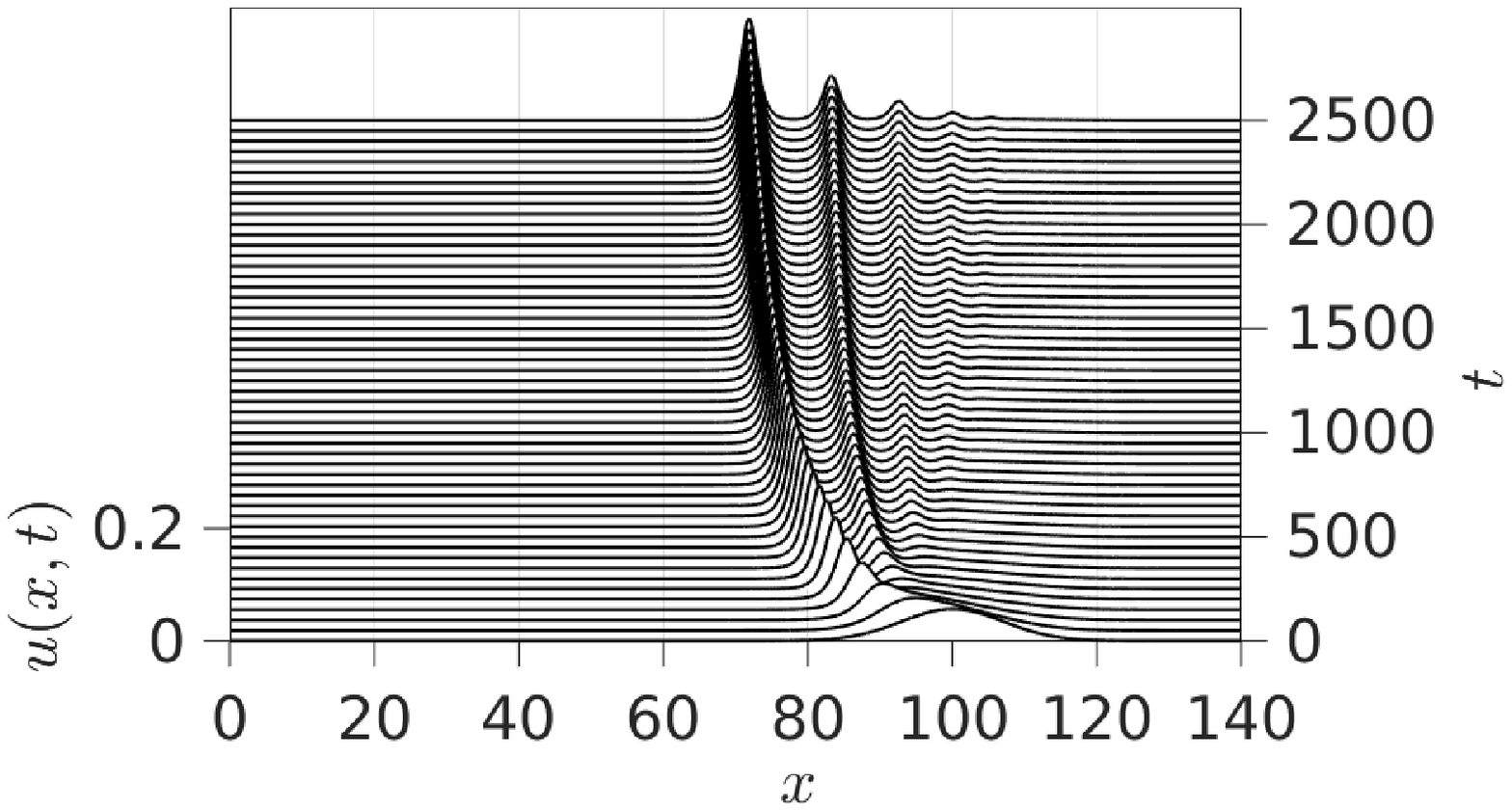}
\hskip -0.5cm
\includegraphics[height=.23\textheight, angle =0]{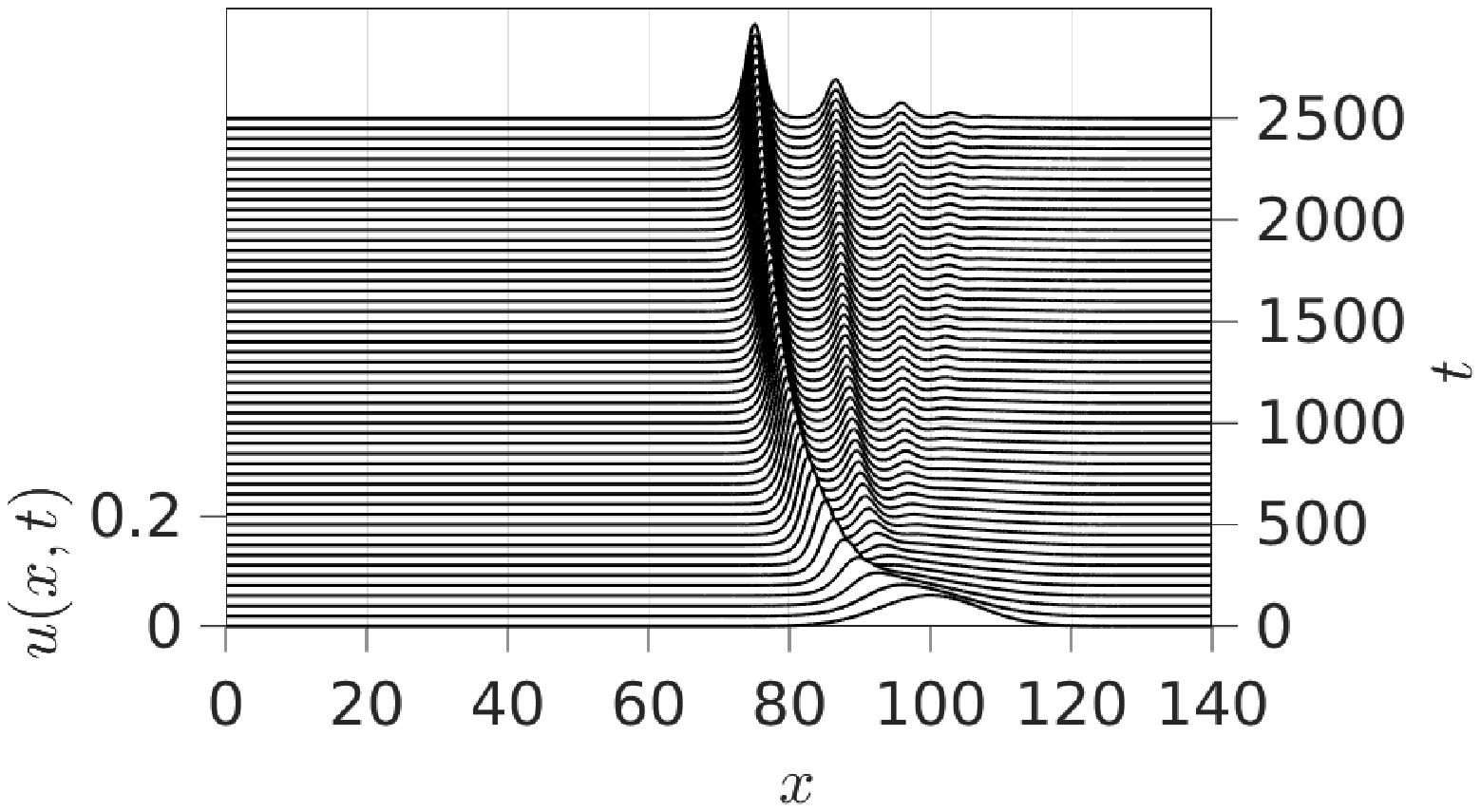}\\
\includegraphics[height=.23\textheight, angle =0]{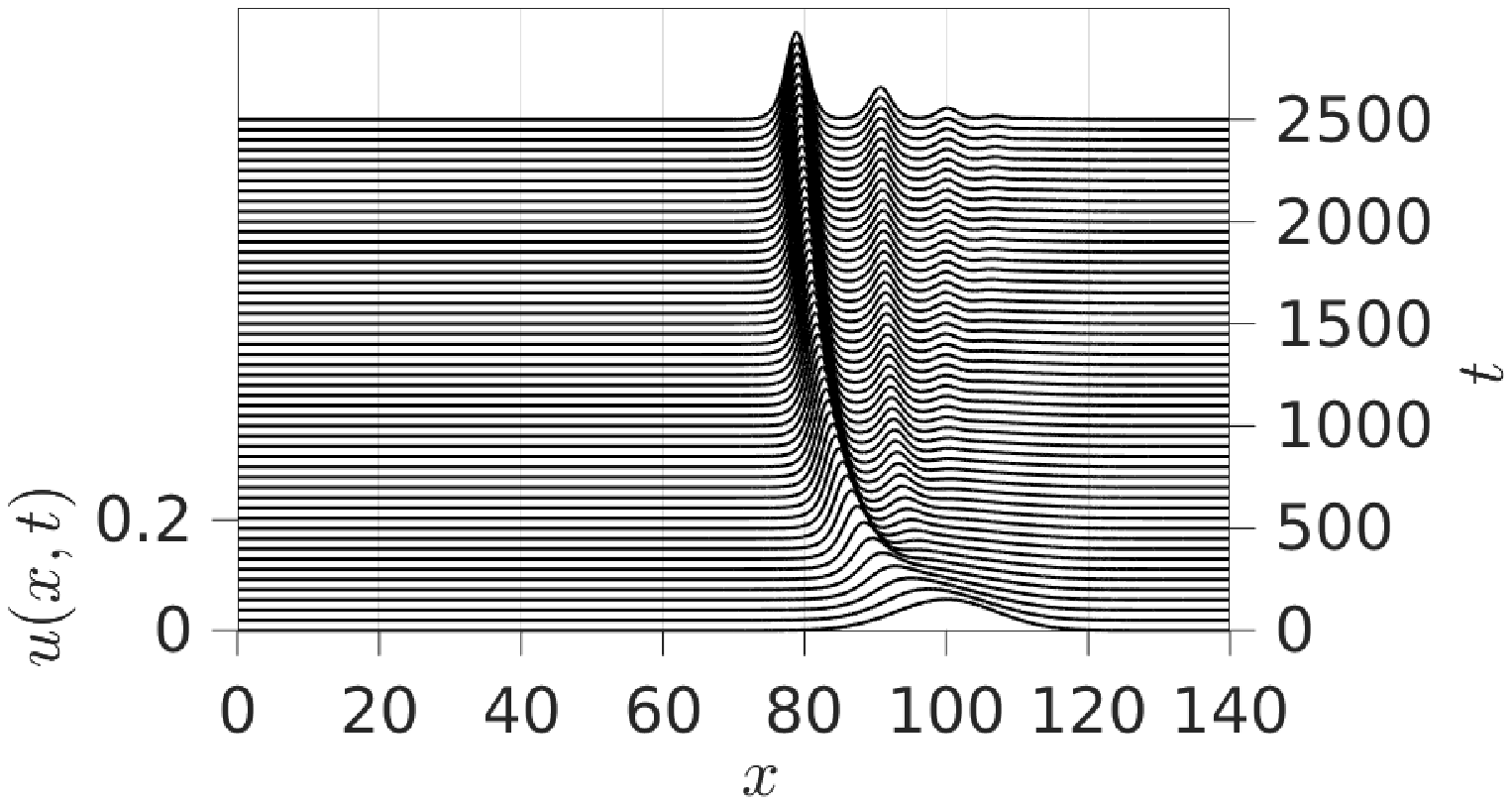}
\hskip -0.5cm
\includegraphics[height=.23\textheight, angle =0]{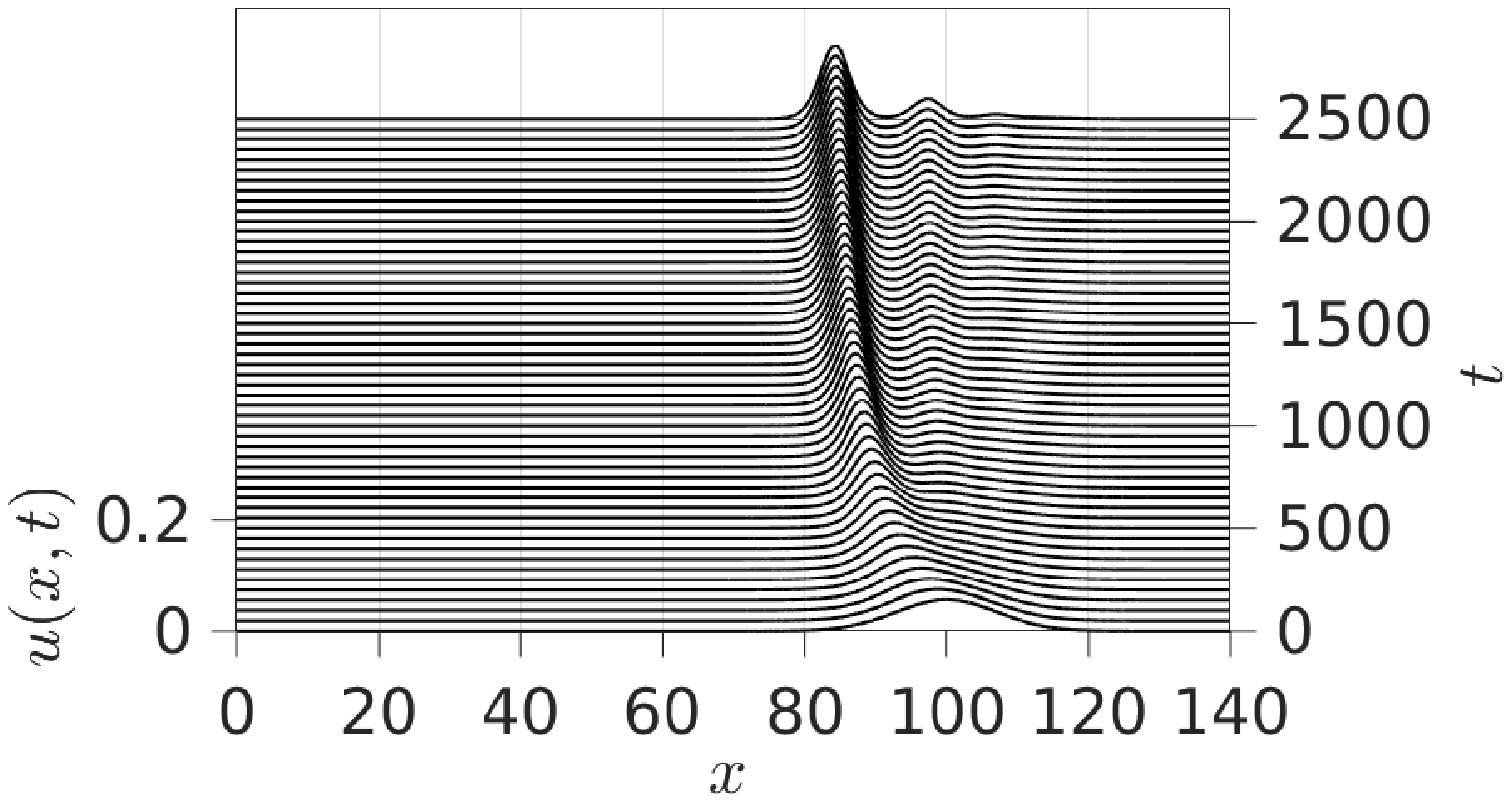}
\end{center}
\caption{
Results on numerical simulations using $N=8192$ Fourier collocation points.
In particular, a Gaussian pulse centered at $x_{0}=100$ with $\sigma=10$ was
used as an initial condition to the $b$-family. Top left and right panels
correspond to values of $b$ of $b=-3$ and $b=-2.5$ whereas the bottom left
and right ones to values of $b$ of $b=-2$ and $b=-1.5$, respectively.
}
\label{fig2}
\end{figure}
\begin{figure}[pt!]
\begin{center}
\includegraphics[height=.23\textheight, angle =0]{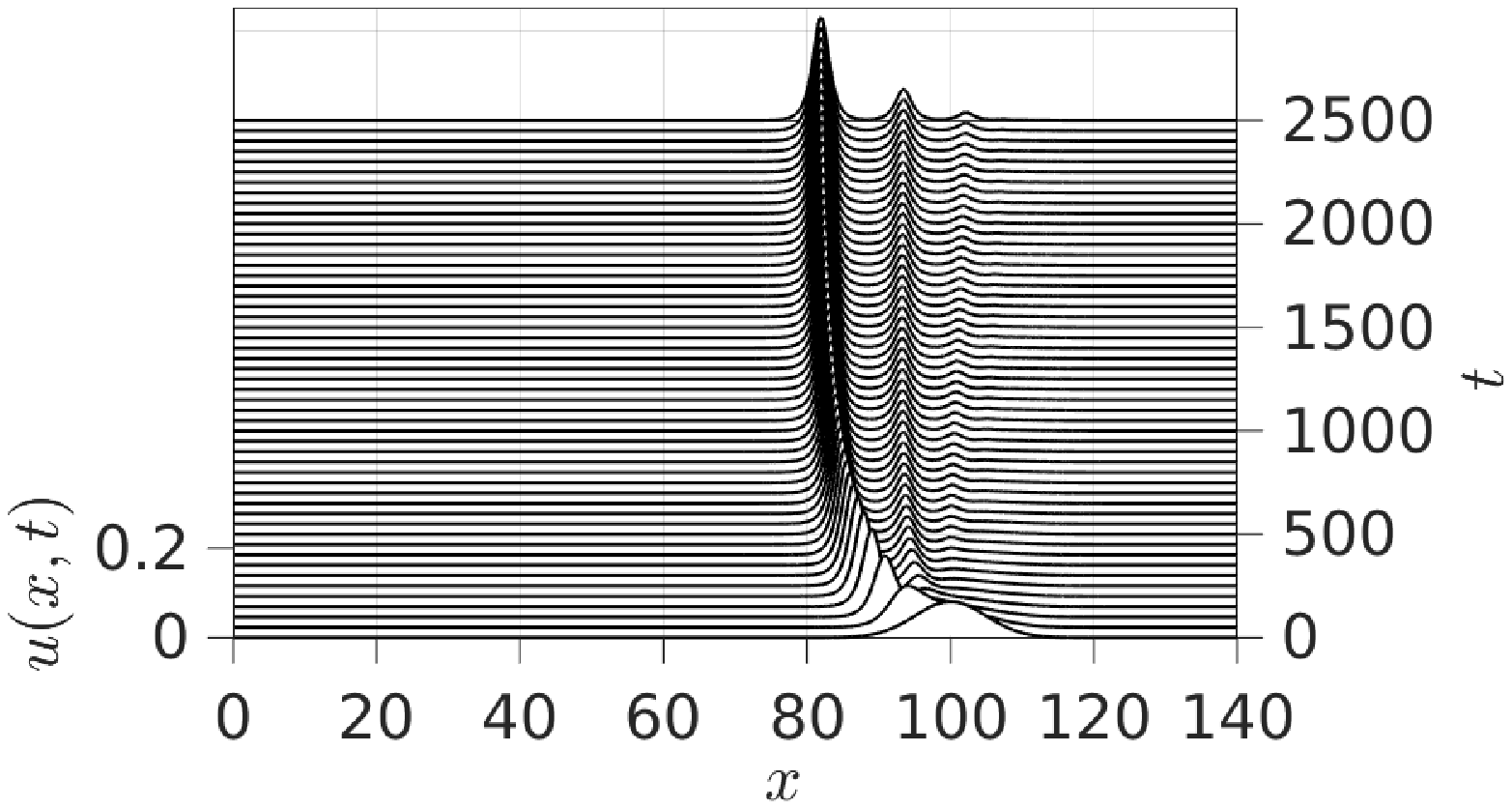}
\hskip -0.5cm
\includegraphics[height=.23\textheight, angle =0]{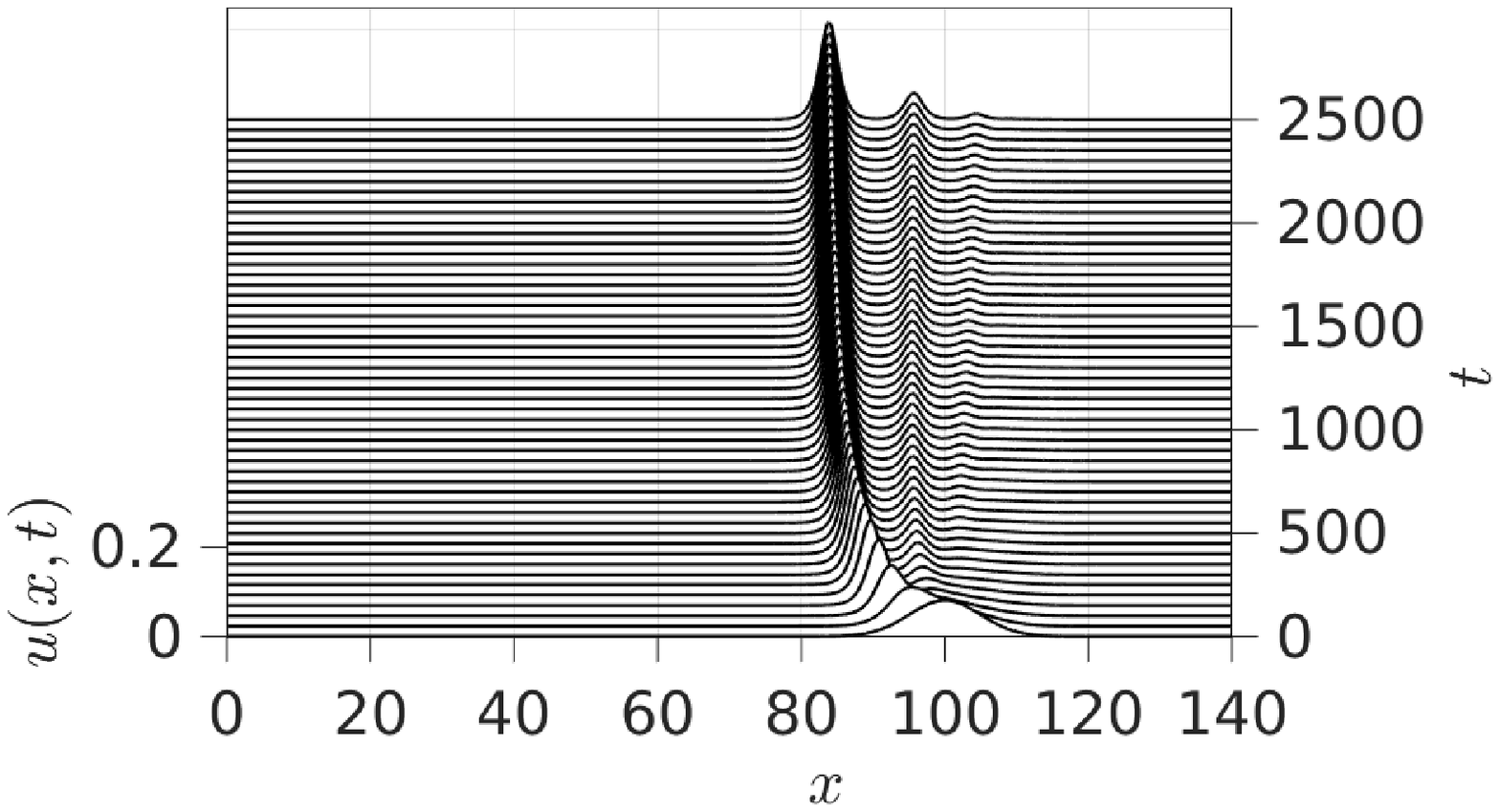}\\
\includegraphics[height=.23\textheight, angle =0]{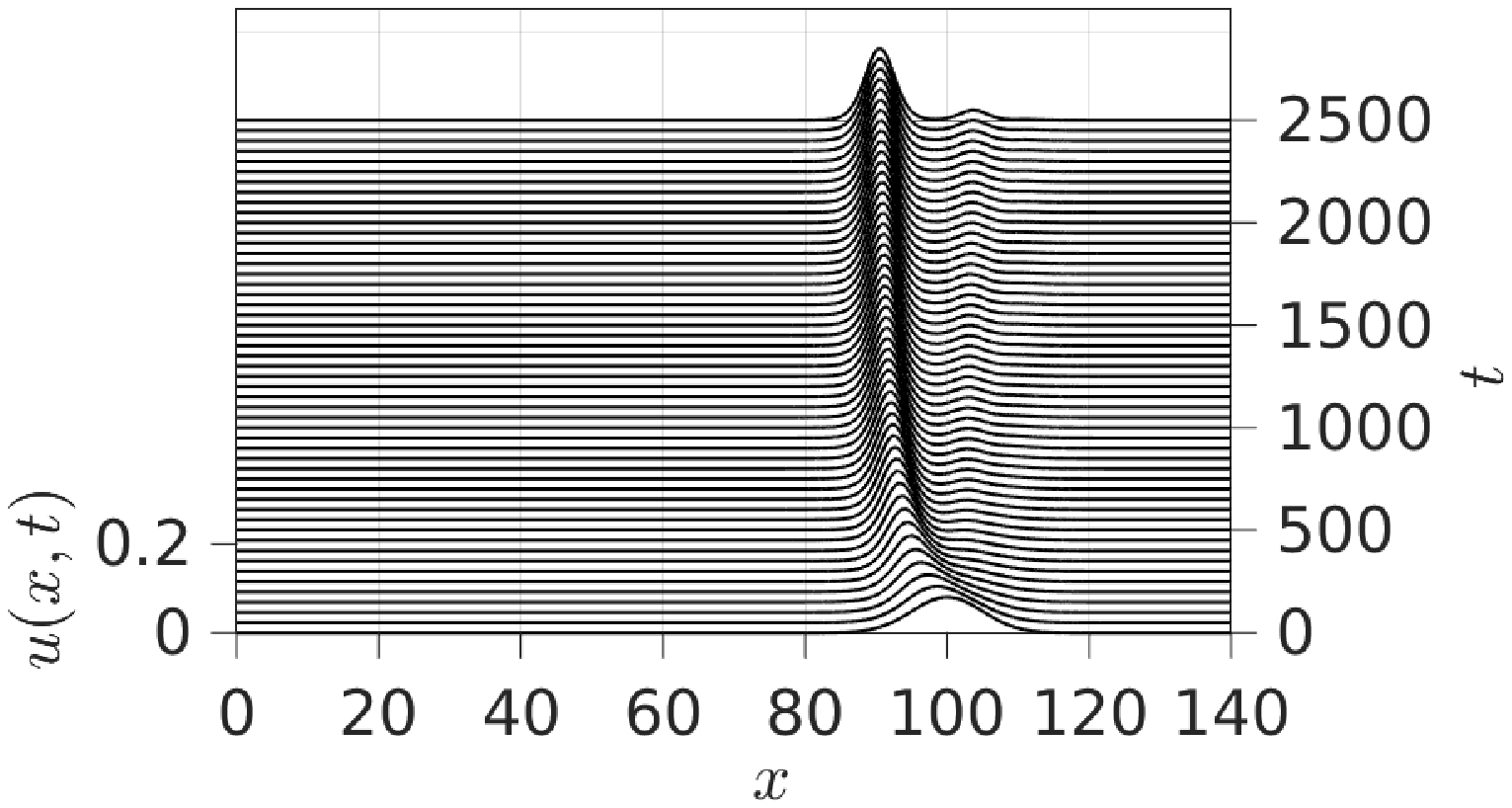}
\hskip -0.5cm
\includegraphics[height=.23\textheight, angle =0]{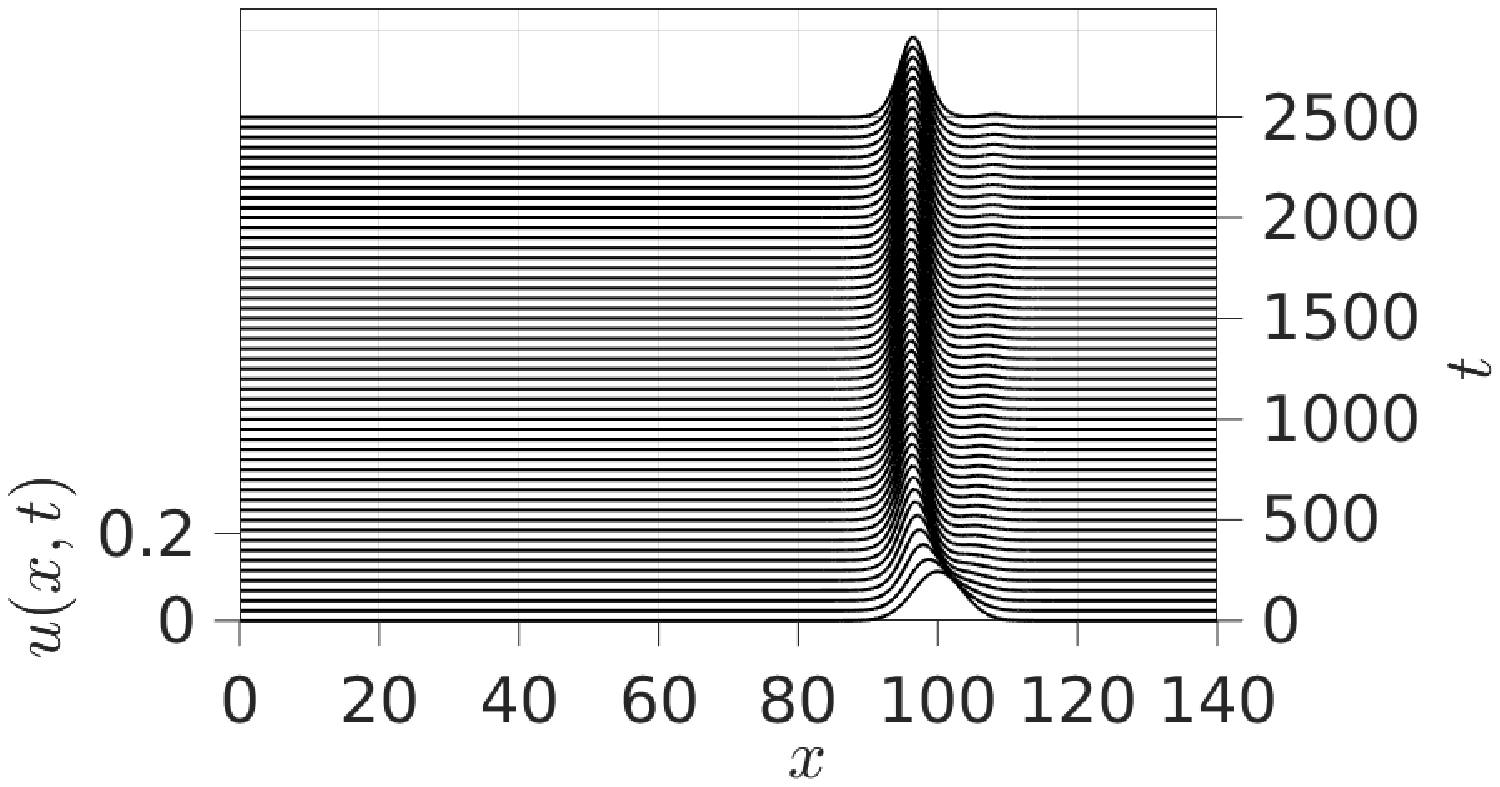}
\end{center}
\caption{
Same as Fig.~\ref{fig2} but for Gaussian initial data with $\sigma=7$ (and
$x_{0}=100$). Top left and right panels correspond to values of $b$ of $b=-3$ 
and $b=-2.5$ whereas the bottom left and right ones to values of $b$ of $b=-2$ 
and $b=-1.5$, respectively.
}
\label{fig3}
\end{figure}
\begin{figure}[pt!]
\begin{center}
\includegraphics[height=.21\textheight, angle =0]{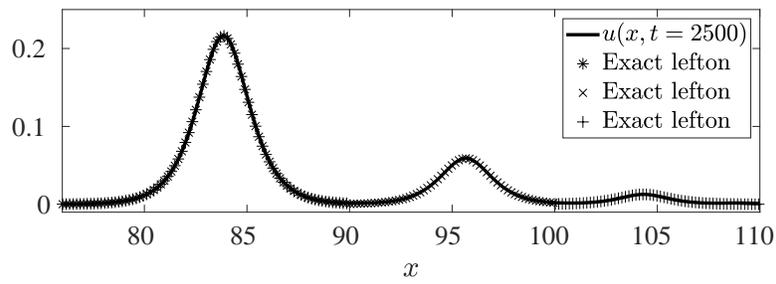}
\end{center}
\caption{
Spatial distribution of the solution of the top right panel 
of Fig.~\ref{fig3} at $t=2500$ (i.e., $b=-2.5$, $\sigma=7$ and 
$x_{0}=100$). The numerically obtained solution is shown with a
solid black line whereas the exact lefton solutions 
[cf. Eq.~\eqref{lefton}] are shown with black stars, crosses 
and plus signs, respectively.
}
\label{fig4}
\end{figure}
\begin{figure}[pt!]
\begin{center}
\includegraphics[height=.23\textheight, angle =0]{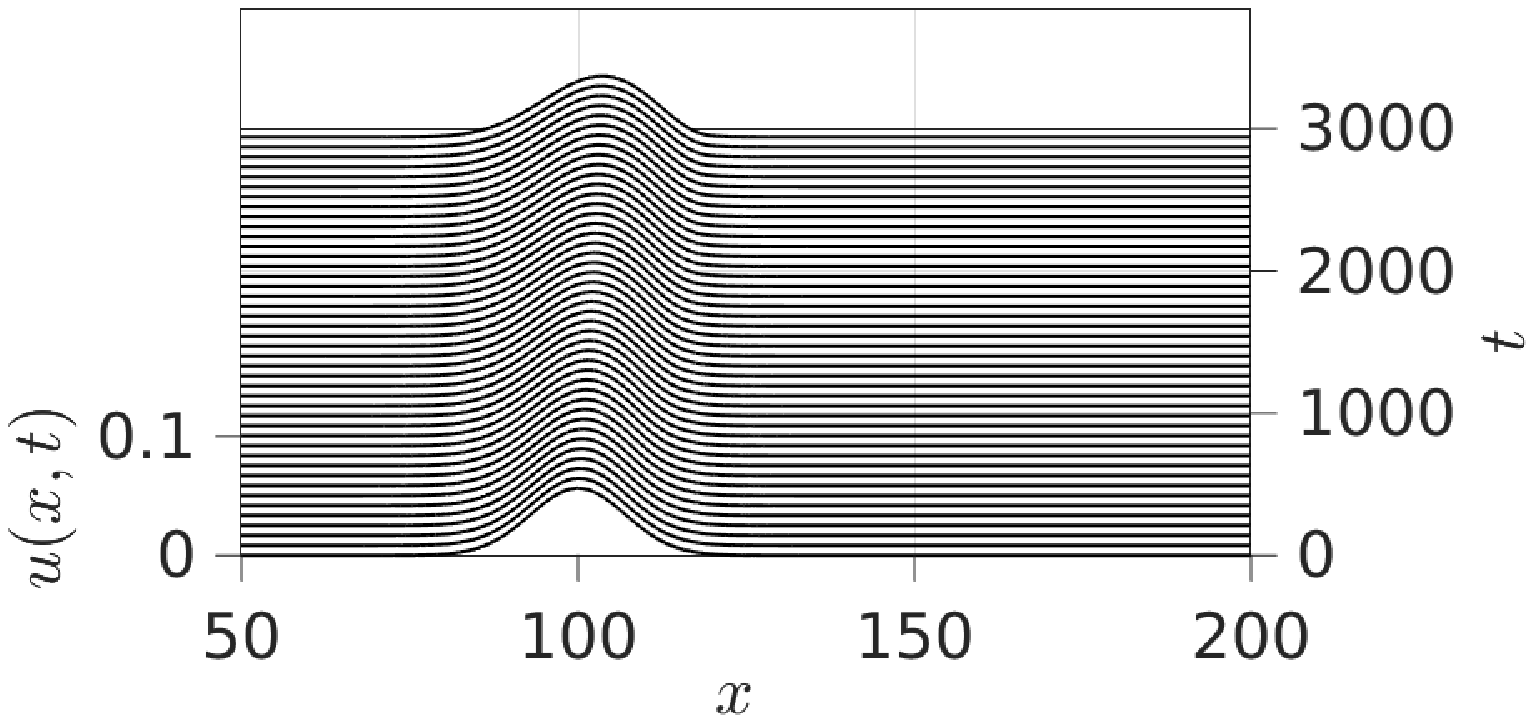}
\hskip -0.5cm
\includegraphics[height=.23\textheight, angle =0]{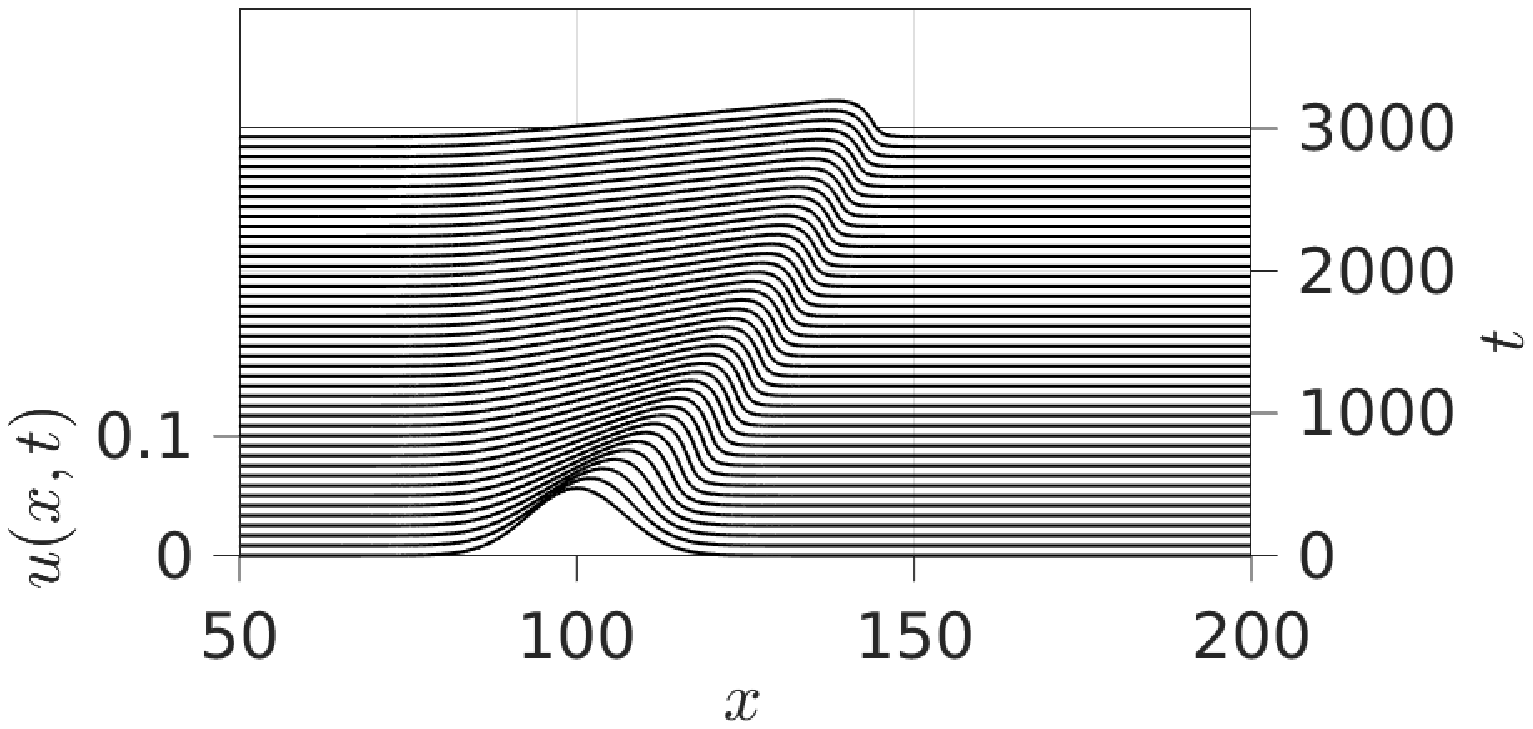}\\
\includegraphics[height=.23\textheight, angle =0]{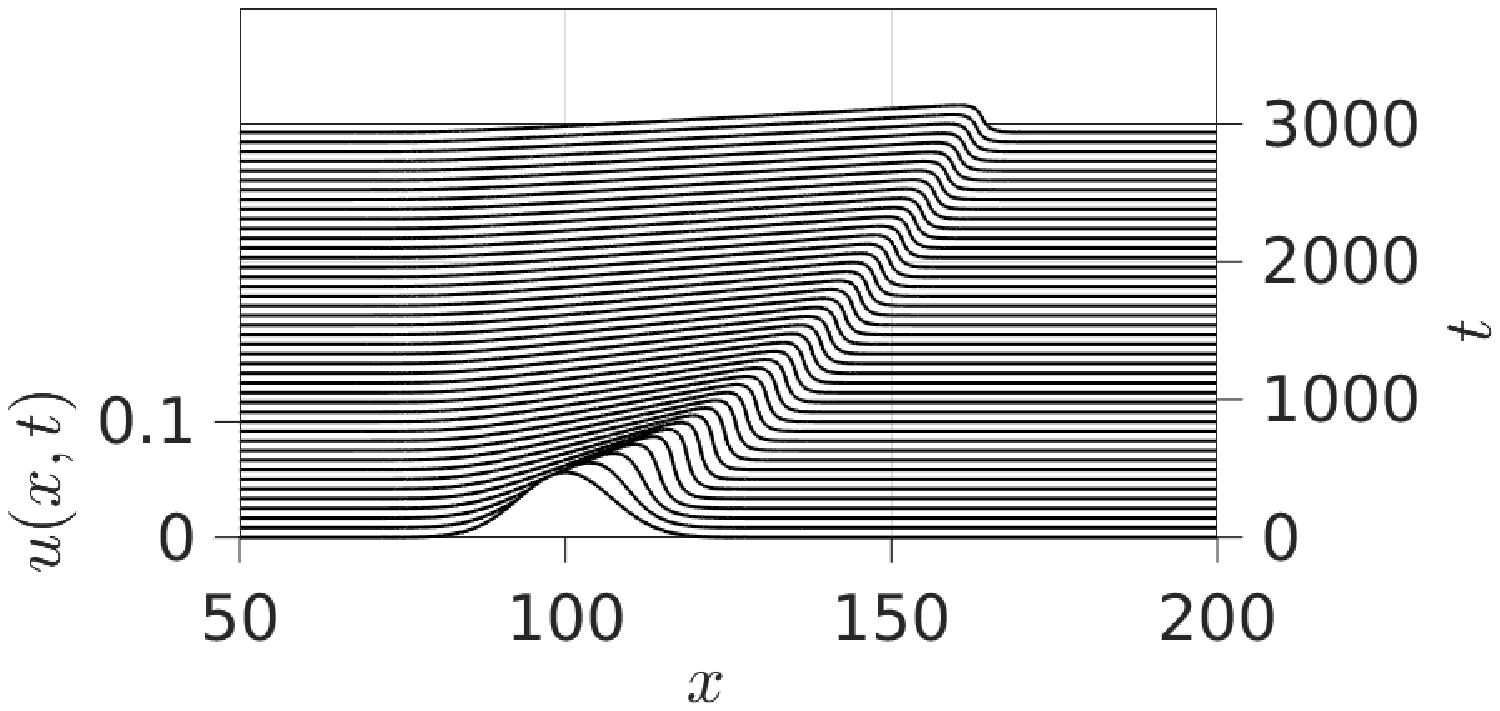}
\hskip -0.5cm
\includegraphics[height=.23\textheight, angle =0]{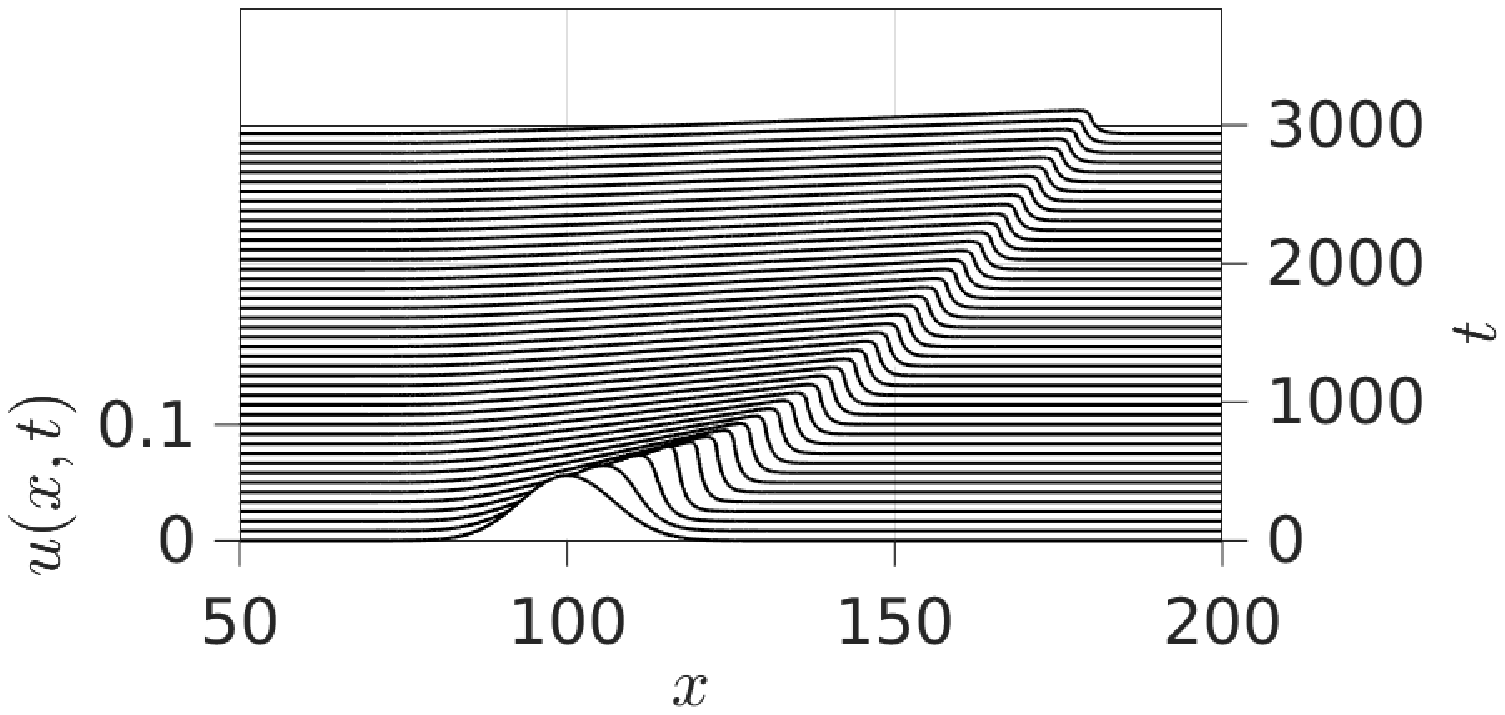}\\
\end{center}
\caption{
Same as Fig.~\ref{fig3} but for Gaussian initial data with $\sigma=10$ (and
$x_{0}=100$) and $N=16384$ Fourier modes. Top left and right panels correspond 
to values of $b$ of $b=-1$ and $b=-0.5$ whereas the bottom left and right ones 
to values of $b$ of $b=0$ and $b=0.5$, respectively. 
}
\label{fig5}
\end{figure}

Next, we focus on the regime $b\in(-1,1)$ in which ramp-cliff solutions
were suggested to be observed from Gaussian initial data. Fig.~\ref{fig5} 
corresponds to numerical results with $\sigma=10$ and $x_{0}=100$ by 
employing $N=16384$ Fourier modes. In particular, the top left and right
panels correspond to the spatio-temporal evolution of $u(x,t)$ with $b=-1$
(i.e., at the bifurcation point) and $b=-0.5$, whereas the bottom left and 
right ones to $b=0$ and $b=0.5$, respectively. From the top left panel of 
Fig.~\ref{fig5} ($b=-1$), it can be discerned that the Gaussian pulse becomes 
slightly wider but represents a nearly stationary solution (see, for example, 
Fig.~5  of~\cite{Holm03}). On the other hand, the top right, bottom left and 
right panels corresponding to $b=-0.5$, $b=0$ and $b=0.5$, respectively, showcase examples of ramp-cliff 
solutions. It should be noted that their amplitude decreases over the time 
evolution although their velocity increases with $b$. 

\bibliographystyle{unsrt}

\begin{thebibliography}{99} 

\bibitem{ch} 
R.~Camassa and D.D.~Holm, %
Phys. Rev. Lett. {\bf 71} (1993) 1661-4. 

\bibitem{ch2} 
R.~Camassa, D.D.~Holm and J.M.~Hyman, %
Adv. Appl. Mech. {\bf 31} (1994) 1-33.  

\bibitem{dp} 
A.~Degasperis and M.~Procesi. %
{\it Symmetry and Perturbation Theory}, 
World Scientific (1999) 23-37.

\bibitem{dhh} 
A.~Degasperis, D.D.~Holm and A.N.W.~Hone, %
Theor. and Math. Phys. {\bf 133} (2002) 1461--72.

\bibitem{Ma05} 
Y.~Matsuno, %
Inverse Problems {\bf{21}} (2005) 2085.

\bibitem{honejpa} 
A.N.W.~Hone, %
J. Phys. A {\bf 32} (1999) L307-L314.

\bibitem{ff}
B.~Fuchssteiner and A.S.~Fokas, %
Physica D {\bf 4} (1981) 47-66. 

\bibitem{wanghone} 
A.N.W.~Hone and J.P.~Wang,  
Inverse Problems {\bf 19} (2003) 129-145. 

\bibitem{miknov} 
A.V.~Mikhailov and V.S.~Novikov. %
J. Phys. A  {\bf 35} (2002) 4775-4790.

\bibitem{honep} 
A.N.W.~Hone, %
\textit{Integrability}, ed. A.V.~Mikhailov, 
Lect. Notes Phys. {\bf 767}, Springer, Berlin, 
Heidelberg (2009) 245-277.

\bibitem{constantin} 
A.~Constantin and D.~Lannes, %
Arch. Rational Mech. Anal.
{\bf 192} (2009) 165-186.

\bibitem{rossen}  
R.I. Ivanov, %
Phil. Trans. R. Soc. A  {\bf 365}  (2007) 2267-2280.  

\bibitem{bm} 
R.~Bhatt and A.V.~Mikhailov, %
On the inconsistency of the Camassa-Holm equation %
with the shallow water theory.
{\tt arxiv:1010.1932v1} 

\bibitem{hh} 
D.D.~Holm and A.N.W.~Hone,
J. Nonlin. Math. Phys. {\bf 12}, Supplement {\bf 1} (2005) 380-94. 

\bibitem{Anco2019} 
S.C.~Anco and E.~Recio, %
J.~Phys.~A {\bf{52}}  (2019) 125203.
 
\bibitem{Guo05} 
B.~Guo and Z.~Liu, %
Chaos, Solitons \& Fractals {\bf{23}} (2005) 1451-1463. 

\bibitem{Holm03a} 
D.D.~Holm and M. F.~Staley, %
Phys. Lett. A {\bf{308}} (2003) 437-444.

\bibitem{Holm03} D.D.~Holm  and M.F.~Staley,   
SIAM J. Appl. Dyn. Syst. {\bf{2}} (2003) 323-380.

\bibitem{tao} 
T. Tao. %
Bull. Amer. Math. Soc. {\bf 46}  (2009) 1-33.  

\bibitem{Hone14} 
A.N.W.~Hone and S.~Lafortune, %
Physica D {\bf{269}} (2014) 28-36. 

\bibitem{Grillakis87} 
M.~Grillakis, J.~Shatah and W.~Strauss,  
J. Functional Analysis {\bf 74} (1987) 160-197.

\bibitem{dhh2} 
A.~Degasperis, D.D.~Holm and A.N.W.~Hone, %
\textit{Proceedings of the Workshop: Nonlinear Physics: Theory 
and Experiment. II}, World Scientific (2002) 37-43.

\bibitem{Gui08} 
G.~Gui, Y.~Liu, and L.~Tian, %
Indiana University Mathematics Journal {\bf{57}} (2008) 1209-1234.

\bibitem{Zhou10} 
Y. Zhou, %
Math. Nachr. {\bf 278}  (2005) 1726-1739. 

\bibitem{Liu06} 
Y.~Liu and Z.~Yin, %
Comm. Math. Phys. {\bf{267}} (2006) 801-820.

\bibitem{Escher08} 
J.~Escher and Z.~Yin, %
Journal f\"ur die reine und angewandte Mathematik (Crelles Journal) %
{\bf{624}}  (2008) 51-80.

\bibitem{katelyn} 
K.~Grayshan, %
Differential and Integral Equations %
{\bf 25},  (2012) 1-20. 

\bibitem{Olver00} 
Y.A.~Li and P.~J.~Olver, %
J. Diff. Eqs. {\bf{162}}  (2000) 27-63.

\bibitem{Rod93} 
G.~Rodr{\'i}guez-Blanco, %
Nonlinear Analysis {\bf{46}} (2001) 309-327.
 
\bibitem{cstrauss} 
A.~Constantin and W.~Strauss, %
Comm. Pure Appl. Math. {\bf 53} (2000) 603--610.  

\bibitem{lin} 
Z.~Lin and Y.~Liu, %
Comm. Pure and Appl. Math. {\bf 62} (2009) 125-146.

\bibitem{Schmudgen12} 
K.~Schm{\"u}dgen, %
{\sl{Unbounded self-adjoint operators on Hilbert space}} %
Graduate Texts in Mathematics Vol.~265. Springer-Verlag (2012).

\bibitem{Pelinovsky2019} 
A.~Geyer and D.~Pelinovsky, %
Proceedings of the American Mathematical Society %
{\bf{148}} (2020)  5109-5125.

\bibitem{Evans87} 
D.~E.~Edmunds and W.~D.~Evans,  %
{\sl{Spectral theory and differential operators}}, %
Oxford University Press (2018).

\bibitem{Kato} 
T.~Kato, %
{\sl{Perturbation theory for linear operators}}, %
Vol.~132, Springer Science \& Business Media (2013).

\bibitem{Renardy} 
M.~Renardy and R.~C.~Rogers, %
{\sl{An Introduction to Partial Differential Equations}}, 
Texts in Applied Mathematics, Springer-Verlag, 2nd edition (2004).

\bibitem{Eveson} 
S.~P.~Eveson, %
Proceedings of the American Mathematical Society %
{\bf{123}} (1995) 3709-3716.

\bibitem{Hitch} 
E.~Di Nezza, G.~Palatucci, and E.~Valdinoci, %
Bulletin des Sciences Math\'ematiques {\bf{136}} (2012) 521-573.

\bibitem{FringHolm} 
O.B.~Fringer and D.D.~Holm, %
Physica D: Nonlinear Phenomena \textbf{150} (2001) 237-263.

\bibitem{jb2001} 
J.~Boyd, %
\textit{Chebyshev \& Fourier Spectral Methods: Second Revised Edition}, %
Dover Books on Mathematics (2001).

\bibitem{mitso2019}
D.C.~Antonopoulos, V.A.~Dougalis and D.E. Mitsotakis, %
Numerische Mathematik \textbf{143} (2019) 833-862. 

\end{thebibliography}

\end{document}